\documentclass{amsart}
\usepackage{bbm}
\usepackage{amsmath}
\usepackage{amsfonts}
\usepackage{amssymb}
\usepackage{amsthm}
\usepackage{mathrsfs}
\usepackage{color}
\usepackage[ USenglish, american]{babel}
\usepackage[T1]{fontenc}
\usepackage[utf8]{inputenc}
\usepackage{enumitem}
\usepackage[toc,page]{appendix}
\usepackage[bookmarks=true,pdfstartview=FitH, pdfborder={0 0 0}, colorlinks=true,citecolor=red, linkcolor=blue]{hyperref}
\usepackage{cleveref}
\usepackage{dcpic,pictexwd}
\crefname{assumption}{assumption}{assumptions} 
\Crefname{assumption}{Assumption}{Assumptions} 

\numberwithin{equation}{section}

\begin{document}
\theoremstyle{plain} 
\newtheorem{theorem}{Theorem}[section]
\newtheorem{corollary}[theorem]{Corollary}
\newtheorem{lemma}[theorem]{Lemma}
\newtheorem{proposition}[theorem]{Proposition}
\newtheorem{definition}[theorem]{Definition} 
\newtheorem{assumption}[theorem]{Assumption} 
\theoremstyle{remark} 
\newtheorem{remark}[theorem]{Remark} 
\newcommand{\bs}{\boldsymbol}
\newcommand{\cou}{\rC([0,t],U)}
\newcommand{\cuor}{\rC_0([0,+\infty),U)}
\newcommand{\cwor}{\rC_0((0,+\infty),\W)}
\newcommand{\cuocuor}{\sL(\rC_0((0,+\infty),U))}
\newcommand{\D}{\mathcal{D}_+-\mathcal{D}_+}
\newcommand{\rL}{\mathrm{L}}
\newcommand{\rC}{\mathrm{C}}
\newcommand{\rW}{\mathrm{W}}
\newcommand{\rH}{\mathrm{H}}
\newcommand{\lu}{\rL^1([0,t],U)}
\newcommand{\lpu}{\rL^p([0,t],U)}
\newcommand{\lpv}{\rL^p([0,t],\V)}
\newcommand{\lprr}{\rL^p([0,+\infty),\mathbb{R}^N)}
\newcommand{\lpru}{\rL^p([0,+\infty),U)}
\newcommand{\lprv}{\rL^p([0,+\infty),\V)}
\newcommand{\ceor}{\rC_0((0,+\infty),\mathbb{R}^N)}
\newcommand{\lru}{\rL^1([0,+\infty),U)}
\newcommand{\lrv}{\rL^1([0,+\infty),\V)}
\newcommand{\lrr}{\rL^1([0,+\infty),\mathbb{R}^N)}
\newcommand{\R}{\mathbb{R}^N}
\newcommand{\F}{{\mathcal{F}}_\infty}
\newcommand{\cxor}{\rC_0((0,+\infty),E)}
\newcommand{\lpe}{\rL^p([0,t],E)}
\newcommand{\ce}{\rC_0((0,t],E)}
\newcommand{\tu}{T([0,t],U)}
\newcommand{\tuor}{T_c([0,+\infty),U)}
\newcommand{\tvor}{T_c([0,+\infty),\V)}
\newcommand{\twoe}{T_c([0,+\infty),\W)}
\parindent=0pt
\labelindent=10pt
\newcommand{\rg}{\operatorname{rg}}
\renewcommand{\Re}{\operatorname{Re}}
\newcommand{\spb}{\operatorname{s}}
\newcommand{\sr}{\operatorname{r}}
\newcommand{\spn}{\operatorname{span}}
\newcommand{\sL}{\mathcal{L}}
\newcommand{\sR}{\mathcal{R}}
\newcommand{\sD}{\mathcal{D}}
\newcommand{\Xmo}{X_{-1}}
\newcommand{\nota}[1]{\textcolor{red}{#1}}
\newcommand{\one}{\mathbbm{1}}
\newcommand{\sBt}{\mathcal{B}_t}
\newcommand{\ub}{\bar u}
\newcommand{\NN}{\mathbb{N}}
\newcommand{\RR}{\mathbb{R}}
\newcommand{\V}{U_1}
\newcommand{\W}{U_\infty}

\newcommand{\At}{\widetilde{A}}
\newcommand{\Xt}{\widetilde{X}}
\newcommand{\Tt}{\widetilde{T}}
\renewcommand{\phi}{\varphi}
\title[Structured Perturbations of Positive Semigroups]{On Structured Perturbations\\ of Positive Semigroups}
\author{Alessio Barbieri$^*$} 
\author{Klaus-Jochen Engel}
\address{University of L’Aquila, Department of Information Engineering, Computer Science and Mathematics, Via Vetoio, Coppito, I-67100 L’Aquila (AQ), Italy}
\email{alessio.barbieri@graduate.univaq.it,klaus.engel@univaq.it}
\thanks{$*$ The author is a member of \textit{Gruppo Nazionale per l'Analisi Matematica, Probabilità e le loro Applicazioni} (GNAMPA) of the Istituto Nazionale di Alta Matematica (INdAM)}

\begin{abstract}
In this note we generalize perturbation results for positive $C_0$-semigroups on AM- and AL-spaces and give a Weiss--Staffans type perturbation result for generators of positive semigroups on Banach lattices. The abstract results are applied to domain perturbations of generators, a heat equation with boundary feedback and perturbations of the first derivative.
\end{abstract}
\maketitle
\section{Introduction}\label{sec:intro}

Many systems evolving in time can be described by an Abstract Cauchy Problem of the form
\begin{equation*}
\tag{ACP}
\label{ACP-1}
\begin{cases}
{\tfrac{d}{dt}}\, x(t)=Gx(t),&t\ge0,\\
x(0)=x_0
\end{cases}
\end{equation*}
where $G$ is an unbounded linear (e.g., differential) operator on a Banach space $X$, cf. \cite[Chap.VI]{EN:00}.  As shown in \cite[Sect.II.6]{EN:00}, this problem is well-posed if and only if $G$ generates a $C_0$-semigroup $(S(t))_{t\ge0}$ on $X$. Moreover, in this case its unique solution is given by $x(t)=S(t)x_0$. 

\smallbreak
For this reason it's vital to have tools at hand which allow to verify the generator property of a given operator $G$.  In the dissipative case the Lumer--Phillips theorem is such a result which provides a characterization of generators of contraction semigroups in terms of $G$ itself and applies to many concrete examples. Its counterpart in the non-dissipative case is the famous Hille--Yosida theorem. 
This result, however, is not based on the given operator $G$ but on growth estimates of \emph{all} powers of its resolvent which frequently are impossible to verify. 
\smallbreak
In order to check well-posedness of \eqref{ACP-1} for (non dissipative) operators $G$ where direct computations involving the resolvent are impossible to perform, one can try to split $G$ into a sum ``$G=A+P\,$'' for a simpler generator $A$ and a perturbation $P$ and then use some kind of perturbation result to conclude that also $G$ generates a $C_0$-semigroup on $X$.

\smallbreak
The Weiss--Staffans theorem on regular linear systems with feedback (cf. \cite[Thms 6.1 \& 7.2]{Wei:94a} and \cite[Sects.7.1 \& 7.4]{Sta:05}) is one of the most general result in this direction. In the present paper we continue our study of purely operator theoretic versions of this theorem initiated in \cite{ABE:14} and \cite{ABE:15} avoiding notions and results from abstract linear systems theory.

\smallbreak
To explain our approach we choose two Banach spaces $X$ and $U$  called \emph{state-}, and \emph{observation/control space}\footnote{The language used here and in the sequel from control theory becomes clearer by considering the linear system~\ref{csu} below. Note that we assume that the control- and observation spaces coincide which in our context is no restriction of generality.}, respectively.
On these spaces we consider
\begin{itemize}
	\item a \emph{state operator} $A:D(A)\subset X \to X$ with non-empty resolvent set $\rho(A)$,
	\item a \emph{control operator} $B\in\sL(U,\Xmo)$, and
	\item an \emph{observation operator} $C\in\sL(Z,U)$.
\end{itemize}
Here, $Z$ is a Banach space satisfying $D(A)\subseteq Z\subseteq X$.
Moreover, $\Xmo$ denotes the extrapolation space with respect to $A$, cf. \cite[Chap.II.5]{EN:00}.
Then we consider structured perturbations of the form $P=BC:Z\to\Xmo$ and obtain the perturbed operator $G$ of the form
\begin{equation}\label{eq:def-G}
G=(A_{-1}+BC)|_X,\quad
D(G):=\bigl\{x\in Z:(A_{-1}+BC)x\in X\bigr\},
\end{equation}
where $A_{-1}:X\subset\Xmo\to\Xmo$ is the extension of $A$ to $X$.

\smallskip
This particular choice of perturbations $P=BC$ is justified by its great flexibility and by many applications, see e.g.\ \Cref{sec:App}. 
Another motivation, which also explains the control theoretic language used above and in the sequel for the operators $A,B,C$ and the spaces $X,U$, goes as follows.

\smallskip
Consider the linear system with control function $u(\cdot):\RR_+\to U$ and observation function $y(\cdot):\RR_+\to U$ given by
\begin{equation*}
\tag*{$\Sigma(A,B,C)$}
\label{csu}
\begin{cases}
\tfrac d{dt}\,x(t)=A_{-1} x(t)+Bu(t),& t\geq0, \\
y(t)=Cx(t),& t\geq0, \\
x(0)=x_0.
\end{cases}
\end{equation*}

If $A$ generates the $C_0$-semigroup $(T(t))_{t\ge0}$, then the solution $x(\cdot)$ of \ref{csu} is formally  given by the variation of parameters formula
\begin{equation*}
x(t)=T(t)x_0+\int_0^tT_{-1}(t-s)Bu(s)\,ds.
\end{equation*}
Closing this system by introducing the feedback $u(t)=y(t)$, one formally obtains the perturbed abstract Cauchy problem
\begin{equation}
\label{eq:acp-BC}
\begin{cases}
\tfrac d{dt}\,x(t)=(A_{-1}+BC)x(t),& t\geq0, \\
x(0)=x_0.
\end{cases}
\end{equation}
Note that the problem~\eqref{eq:acp-BC} is not formulated in the original state space $X$ but in the
extrapolation space $\Xmo$. In most applications this space is only of theoretical interest and in general difficult to determine explicitly. However, being interested in the dynamics on the original state space we restrict this problem to $X$ and consider the part $G$ of $A_{-1}+BC$ in $X$  as defined in \eqref{eq:def-G}. This finally gives \eqref{ACP-1} in the state space $X$ for the operator 
\begin{equation}\label{eq:def-A_BC-i}
A_{BC}:=G=(A_{-1}+BC)|_X.
\end{equation}
As shown in \cite[Sects.4.1 \& 4.2]{ABE:14}, Weiss--Staffans perturbations of the type $P=BC$ cover in particular Miyadera--Voigt- (where $U=X$ and $B=Id$ and $Z=D(A)$) and Desch--Schappacher perturbations (where $U=X$ and $C=Id$)  which are, e.g.,  studied in detail in \cite[Sect.III.3]{EN:00}.

\smallbreak
For these two perturbation results there exist simplified versions due to Desch--Voigt and Batkai et al., respectively, in case $A$ generates a positive semigroup on a Banach lattice $X$, and the perturbation 
\begin{itemize}
\item $P=C:D(A)\to X$ is positive and $X$ is an AL-space, see \cite{Des:88} and \cite{Voi:89}, and
\item $P=B:X\to\Xmo$ is positive and $X$ is an AM-space, see \cite[Thm.1.2]{BJVW:18},
\end{itemize}
respectively. More precisely, in these works a spectral condition concerning the input-output operator $\F$ of the system \eqref{csu} mapping the input $u(\cdot)$ to its output $y(\cdot)$ is replaced by a resolvent condition on $CR(\lambda,A_{-1})B$ which in many applications is much easier to verify.

\smallbreak
The aim of this note is twofold. First we generalize these simplified versions to operators $P=BC$ where instead of the previous setting we impose that
\begin{itemize}
\item $B:U\to X$ and $C:D(A)\to U$ are positive and $U$ is an AL-space, 
or
\item $B:U\to\Xmo$ and $C:X\to U$ are positive and $U$ is an AM-space, 
\end{itemize}
respectively, where in both cases $X$ can be an arbitrary Banach lattice.
Secondly, we give a simplified version of the Weiss-Staffans type perturbation result \cite[Thm.10]{ABE:14} in case $A$ generates a positive semigroup on a Banach lattice $X$, and
\begin{itemize}
\item $B:U\to\Xmo$ and $C:Z\to U$ are positive and $U=\R$, i.e., simultaneously is  of both types AL and AM.
\end{itemize}

\smallbreak
This paper is organized as follows. In \Cref{GS} we introduce our general setup and prove an admissibility result for the control operator $B$. \Cref{MR} contains the main results of this note: 
\begin{itemize}
\item \Cref{AM} on positive structured perturbations factorized via AM-spaces,
\item \Cref{AL} on positive structured perturbations factorized via AL-spaces
\item \Cref{PT} on positive structured perturbations factorized via $\R$, and
\item \Cref{thm:Dom} on structured perturbation via domination.
\end{itemize}
In \Cref{sec:App} we show the power of our approach and apply these results to boundary perturbations of domains of generators, a heat equation with boundary feedback and perturbations of the first derivative.  These examples show in particular the great advantage it makes assuming that $U$ instead of $X$ has to be an AM- or AL-space as needed in the previous works \cite{Des:88,Voi:89,BJVW:18}.

\smallbreak
For related results on perturbation of generators of positive $C_0$-semigroups we refer to \cite{Are:87b,AR:91,BA:06,GTK:19,EGa:23}.

\section{The General Setup}\label{GS}

To start with we introduce the main objects of our investigations.

\begin{assumption}\label[assumption]{MA}
If not stated otherwise, in the sequel we consider
\begin{enumerate} [label=(\roman*)]
\item a Banach lattice $X$ called the \emph{state space};
\item a Banach lattice $U$ called the \emph{control/observation space};
\item a \emph{state operator} $A:D(A)\subset X\to X$ generating a positive $C_0$-semigroup $(T(t))_{t\geq0}$ on $X$;
\item a \emph{control operator} $B:U\to\Xmo$; 
\item an  \emph{observation operator} $C:Z\to U$, where $Z$ is a vector space satisfying
\begin{equation}\label{incl}
D(A)\subseteq Z\subseteq X.
\end{equation}
\end{enumerate}
\end{assumption} 

Here the extrapolation space $\Xmo$ of $X$ associated to the operator $A$ is given by the completion $(X,\lVert\cdot\rVert_{-1})^{\sim}$ where $\lVert x\rVert_{-1}:=\lVert R(\lambda,A)x\rVert$ for some fixed $\lambda\in\rho(A)$. Then for any $t\ge0$ the operator $T(t)$ possesses a unique continuous extension $T_{-1}(t)\in\sL(\Xmo)$ and $(T_{-1}(t))_{t\ge0}$ is a $C_0$-semigroup on $\Xmo$ with generator $A_{-1}$ satisfying $D(A_{-1})=X$. This definition does not depend on the particular choice of $\lambda\in\rho(A)$ since the norms $\lVert R(\lambda,A)\cdot\rVert$ are equivalent for any $\lambda\in\rho(A)$. For a detailed treatment of these facts we refer to \cite[Chap.II.5]{EN:00}.

\smallbreak
In our case $X$ is a Banach lattice, and it is a priori not clear how to extend the concept of positivity to $\Xmo$. If $A$ generates a positive  $C_0$-semigroup on $X$, we follow \cite[Def.2.1]{BJVW:18} and define the positive cone in $\Xmo$ by
\begin{equation*} 
X_{-1,+}:=\overline{X_+}^{\,\|\cdot\|_{-1}},
\end{equation*}
where $X_+$ denotes the positive cone of $X$ (see Appendix \ref{appa}). It is immediate that $X_+\subset X_{-1,+}$. If $X$ is a real Banach lattice, then  by \cite[Prop.2.3]{BJVW:18} we also have $X_{+}=X_{-1,+}\cap X$. Moreover, by \cite[Rem.2.2]{BJVW:18} the following holds.

\begin{lemma} \label{L:B-pos}
Let $A$ be the generator of a positive semigroup on the Banach lattice $X$.
Then the operator $B:U\to \Xmo$ is positive, i.e.\  $BU_+\subseteq X_{-1,+}$,  if and only if $R(\lambda,A_{-1})B:U\to X$ is positive for all $\lambda>\spb(A)$.
\end{lemma}

Throughout this note we denote by $|\cdot|_{\R}$ an arbitrary norm on $\R$, in contrast to the absolute value which will be denoted by $|\cdot|$. In case we have to be more specific, $|\cdot|_p$ denotes the $p$-norm on $\R$ for $1\le p\le\infty$.  Finally, if $U=\R$ and $p=1$ or $p=\infty$ we use the notations $\V:=(\R,\lvert\cdot\rvert_1)$ and $\W:=(\R,\lvert\cdot\rvert_\infty)$.

\subsection{Admissibility}\label{ssec:adm}

In this subsection we recall some standard notions from linear systems theory, cf.\ \cite[Chap.10]{Sta:05} and the references therein, which are essential for our approach.

\begin{definition}
The control operator $B:U\to\Xmo$ is said to be $p$-\emph{admissible} for $1\le p<\infty$ if there exists $t>0$ such that
\begin{equation}\label{boo}
\int_{0}^{t}T_{-1}(t-s)Bu(s)\,ds\in X\quad\text{for all}\,\,u\in\lpu.
\end{equation}
\end{definition}

This means that the \emph{controllability map} $\mathcal{B}_{t}:\lpu\to \Xmo$ given by
\begin{equation}\label{btpo}
\mathcal{B}_{t}u:=\int_{0}^{t}T_{-1}(t-s)Bu(s)ds,\quad u\in\lpu
\end{equation}
has range $\rg(\mathcal{B}_{t})\subseteq X$.  Since the operator $\mathcal{B}_t:\lpu\to \Xmo$ is bounded and $X$ is continuously embedded in $\Xmo$, by the Closed Graph Theorem it follows that $\mathcal{B}_t:\lpu\to X$ is bounded (see \cite[Cor.B.7]{EN:00}).

\smallbreak

In case $p=\infty$ we give the following definition.

\begin{definition}
The control operator $B:U\to\Xmo$ is said to be $\infty$-\emph{admissible} if there exists $t>0$ such that
\begin{equation}\label{boo1}
\int_0^{t}T_{-1}(t-s)Bu(s)ds\in X\quad\text{for all}\,\,u\in\cou.
\end{equation}
\end{definition}

As before, this definition can be equivalently rephrased by saying that the controllability map $\mathcal{B}_{t}:\cou\to \Xmo$ given by
\begin{equation}\label{bto}
\mathcal{B}_{t}u:=\int_{0}^{t}T_{-1}(t-s)Bu(s)ds,\quad u\in\cou
\end{equation}
has range $\rg(\mathcal{B}_{t})\subseteq X$. By the same reasoning as above, this implies that $\mathcal{B}_t:\cou\to X$ is bounded.
In \Cref{pDS} we will verify the $\infty$-admissibility of a control operator $B$ by using the following result.

\begin{proposition}\label{dens}
The control operator $B:U\to\Xmo$ is $\infty$-admissible if
there exists $t>0$, a subspace $\mathcal{D}\subset \rL^\infty([0,t],U)$ and $M\ge0$ such that $\cou\subseteq\overline{\mathcal{D}}$ and
\begin{equation*} 
\begin{aligned}
	&(i)\quad \int_{0}^{t}T_{-1}(t-s)Bu(s)\,ds\in X\,\,&&\text{for all}\,\,u\in\mathcal{D},\\
	&(ii)\quad \left\lVert\int_{0}^{t}T_{-1}(t-s)Bu(s)\,ds\right\rVert_X\le M\lVert u\rVert_\infty\,\,&&\text{for all}\,\,u\in\mathcal{D}.
\end{aligned}
\end{equation*}
Moreover, in this case $\|\sBt\|_{\sL(\cou,X)}\le M$.
\end{proposition}

\begin{proof}
In any case, we can define the bounded linear operator $\sBt:\rL^\infty([0,t],U)\to\Xmo$ as in \eqref{bto} for $u\in \rL^\infty([0,t],U)$. Now assumptions (i) and (ii) imply that the restriction $\sBt|_{\sD}:\sD\to X$ is bounded of bound $M$, hence possesses a unique bounded extension $\sR:\overline{\sD}\to X$ of the same bound.
Since by assumption $\cou\subseteq\overline{\mathcal{D}}$, for every $u\in\cou$ there exists a sequence $(u_n)_{n\in\NN}\subset\sD$ converging to $u$ in $\rL^\infty([0,t],U)$. Then $\sBt u_n\to \sBt u$ in $\Xmo$ and $\sBt u_n=\sR u_n\to\sR u$ in $X$ as $n\to+\infty$. This implies $\sBt u=\sR u\in X$ for all $u\in\cou$, i.e., $B$ is $\infty$-admissible.
\end{proof}

\begin{remark}\label{badm}
Condition \eqref{boo} becomes less restrictive as $p\in[1,\infty)$ grows. In particular, \eqref{boo1} is weaker than \eqref{boo} for any $p\in[1,\infty)$.
Moreover, a bounded control operator $B:U\to X$ is always $1$-admissible, hence also $p$-admissible for every $p\in(1,+\infty]$. 
\end{remark}

For observation operators we recall an analogous notion
\begin{definition}
The observation operator $C:Z\to U$ is said to be $p$-\emph{admissible} for $1\le p<\infty$ if there exists $t>0$ and $M\ge0$ such that
\begin{equation}\label{coo}
\int_{0}^{t}\bigl\| CT(s)x\bigr\|_{U}^p\,ds\le M\cdot\lVert x\rVert_X^p\quad\text{for all}\,\,x\in D(A).		
\end{equation}
\end{definition}

Since $D(A)$ is dense in $X$, the previous definition implies the existence of a bounded \emph{observability map} $\mathcal{C}_t:X\to\lpu$ satisfying $\lVert\mathcal{C}_t\rVert\le M$ and
\begin{equation}\label{cto}
(\mathcal{C}_t x)(s):=CT(s)x,\quad x\in D(A),\,\,s\in[0,t].
\end{equation}

Note that condition \eqref{coo} becomes more restrictive as $p\in[1,\infty)$ grows.

\begin{remark}\label{cadm}		
For the operator $C:Z\to U$ one could also introduce the concept of $\infty$-admissibility, i.e.,  ask for the existence of $t>0$ and $M\ge0$ satisfying
\[\sup_{s\in[0,t]}\bigl\|CT(s)x\bigr\|_U\le M\lVert x\rVert_X\quad\text{for all $x\in D(A)$}.\]
However, it is easy to see that $C$ is $\infty$-admissible if and only if $C\in\sL(X,U)$. In this case, the above estimate extends by density to all $x\in X$.
\end{remark}

A last notion we need in the sequel is the \emph{compatibility} of the triple $(A,B,C)$.

\begin{definition}\label{def2}
The triple $(A,B,C)$ of a state operator $A:D(A)\subseteq X\to X$, a control operator $B:U\to\Xmo$ and an observation operator $C:Z\to U$ is said to be \emph{compatible} if for some $\lambda\in\rho(A)$ one has $\rg(R(\lambda,A_{-1})B)\subset Z$.
\end{definition}

If this range condition is satisfied for some $\lambda\in\rho(A)$ then by the resolvent identity it holds for all $\lambda\in\rho(A)$.
Since any everywhere defined positive operator between Banach lattices is bounded (see, e.g., \cite[Thm.10.20]{BKR:17}), \eqref{incl}  and \Cref{L:B-pos} imply that for every positive compatible triple $(A,B,C)$ we have
\begin{equation*} 
CR(\lambda,A_{-1})B\in\sL(U)\quad\text{for all $\lambda>\omega_0(A)$}.
\end{equation*}
Here we call $(A,B,C)$ positive if $A$ generates a positive semigroup and $B$, $C$ are positive.

\section{The Main Results}\label{MR}

\subsection{Positive Structured Perturbations factorized via AM-Spaces}\label{pDS}
In this section we give a generalization of a positive perturbation theorem due to Bátkai et al. in \cite{BJVW:18}, cf. \Cref{rem:DS-pos}. With respect to our setting we introduced in \Cref{GS} we assume throughout this subsection that $D(C)=Z=X$ which implies that $C\in\sL(X,U)$ is bounded.

\begin{theorem}\label{AM}
Let $A:D(A)\subseteq X\to X$ be the generator of a positive  $C_0$-semigroup $(T(t))_{t\ge0}$ on a Banach lattice $X$. Assume that $B:U\to\Xmo$ and $C:X\to U$ are positive linear operators, where $U$ is an AM-space. If the spectral radius $\sr(CR(\lambda,A_{-1})B)<1$ for some $\lambda>\omega_0(A)$, then the operator
\begin{equation}\label{abcam}
	\begin{aligned}
\hspace{3cm} A_{BC}&:=\left(A_{-1}+BC\right)|_{X},\\
D(A_{BC})&:=\bigl\{x\in X: A_{-1}x+BCx\in X\bigr\}
 \end{aligned}
 \end{equation}
generates a positive  $C_0$-semigroup $(S(t))_{t\ge0}$ on $X$ satisfying the variation of parameters formula
\begin{equation}\label{VP1}
S(t)x:=T(t)x+\int_{0}^{t}T_{-1}(t-s)\cdot BC\cdot S(s)x\,ds,\quad x\in D(A_{BC}),\,\,t\ge0.
\end{equation}
\end{theorem}

\begin{remark}\label{rem:DS-pos}
As already mentioned, this result generalizes \cite[Thm.1.2]{BJVW:18} where it is assumed that $X=U$ and $C=Id$. In particular, in \cite{BJVW:18}  one needs $X$ to be an AM-space while in \Cref{AM} only the boundary space $U$ has to be of the type AM while $X$ is allowed to be an arbitrary Banach lattice. See \Cref{subsec:HE} for an example where this generalization is crucial.
\end{remark}

In order to prove this theorem we need some preparation. First, by rescaling and in virtue of \cite[Sect.II.2.2]{EN:00} we can in the sequel always assume that $\omega_0(A)<0$ and take $\lambda=0$.

\smallbreak
Next, under the assumptions of \Cref{AM} the observation operator $C:X\to U$ is bounded, hence $\infty$-admissible for all $t>0$, cf.\ \Cref{cadm}. The following result shows that the control operator $B:U\to\Xmo$ is $\infty$-admissible as well. To verify this claim we use the spaces $\cuor$ of all continuous $U$-valued functions vanishing at infinity endowed by the $\infty$-norm and $\tu$
 of all $U$-valued step functions defined on $[0,t]$ for some $t>0$. Note that the latter is a normed sublattice of $\rL^\infty([0,t],U)$, and its closure in $\rL^\infty([0,t],U)$ contains $\cou$. 

\begin{lemma}\label{binf}
If $B:U\to\Xmo$ is a positive linear operator, where $U$ is an AM-space, then $B$ is $\infty$-admissible for any $t>0$ and $\mathcal{B}_t:\cou\to X$ defined by \eqref{bto} is positive and satisfies
\begin{equation}\label{bt}
	\lVert\mathcal{B}_t\rVert_{\sL(C,X)}\le\bigl\lVert A_{-1}^{-1}B\bigr\rVert_{\sL(U,X)}\quad\text{for all}\,\,t>0.
\end{equation}
\end{lemma}

\begin{proof} This proof follows in part the ones of \cite[Prop.4.2 and Lem.4.3.(ii)]{BJVW:18}.
We verify the conditions (i) and (ii) of \Cref{dens} for $\mathcal{D}=\tu$. Fix $t>0$ and take $u\in\tu$, i.e.,
\[u(s):=\sum_{n=1}^{N}\one_{I_n}(s)\,u_n,\quad s\in[0,t],\]
where $u_1,\dots,u_N\in U$, $I_1,\dots,I_N\subseteq[0,t]$ are pairwise disjoint intervals such that $\dot\bigcup_{n=1}^N I_n=[0,t]$ and $\one_{I_n}$ is the characteristic function of $I_n$. Then
\begin{equation*}
\begin{split}
\int_{0}^{t}T_{-1}(t-s)Bu(s)\,ds
&=\sum_{n=1}^{N}\int_{t_{n-1}}^{t_n}T_{-1}(t-s)Bu_n\,ds,
\end{split}
\end{equation*}
where $t_{n-1}\leq t_n$ denote the endpoints of $I_n$. Since
\begin{equation}\label{eq:Bt-sf-cont}
\int_{t_{n-1}}^{t_n}T_{-1}(t-s)Bu_n\,ds
=\bigl(T(t-t_{n-1})-T(t-t_n)\bigr)A_{-1}^{-1}Bu_n\in X
\end{equation}
we obtain (i), i.e., $\mathcal{B}_t:\tu\to X$ is well-defined. 

\smallbreak
In order to prove condition (ii), we first assume that $0\le u\in\tu$, i.e., $u_n\ge0$ for all $n=1,\dots,N$. Next define
\[\ub:=\sup_{n=1,\dots,N}u_n\ge0\]
which satisfies $0\le u(s)\le \ub$ for all $s\in[0,t]$. Using the positivity of $\lambda R(\lambda,A_{-1})B$ for $\lambda>\max\{0,\spb(A)\}$ and $(T(t))_{t\ge0}$ we conclude
\begin{equation*}
\begin{split}
0&\le \int_{0}^{t}T(t-s)\lambda R(\lambda,A_{-1})Bu(s)\,ds
\le \int_{0}^{+\infty}T(s)\lambda R(\lambda,A_{-1})B\ub\,ds\\
&=-\lambda R(\lambda,A)A_{-1}^{-1}B\ub
\end{split}
\end{equation*} 
for all $\lambda>\max\{0,\spb(A)\}$. 
Since $\lambda R(\lambda,A)\to Id$ in $X$ and, by \cite[Prop.A.3]{EN:00},
\[
\lambda R(\lambda,A_{-1})T_{-1}(t-s)Bu(s)\to T_{-1}(t-s)Bu(s)
\]
uniformly for $s\in[0,t]$ in $X_{-1}$  as $\lambda\to+\infty$ this implies
\begin{equation}\label{eq:est-Bt}
0\le\mathcal{B}_t u=\int_{0}^{t}T_{-1}(t-s)Bu(s)\,ds\le-A_{-1}^{-1}B\ub.
\end{equation}
Hence,
\begin{equation*}
\begin{split}
\left\lVert \mathcal{B}_tu\right\rVert_X
&\le\bigl\lVert A_{-1}^{-1}B\ub\bigr\rVert_X
\le\bigl\lVert A_{-1}^{-1}B\bigr\rVert_{\sL(U,X)}\cdot\Bigl\lVert\sup_{n=1,\dots,N}u_n\Bigr\rVert_{U}\\
&=\bigl\lVert A_{-1}^{-1}B\bigr\rVert_{\sL(U,X)}\cdot\sup_{n=1,\dots,N}\lVert u_n\rVert_{U}
=\bigl\lVert A_{-1}^{-1}B\bigr\rVert_{\sL(U,X)}\cdot\lVert u\rVert_{\infty},
\end{split}
\end{equation*} 
where in the second line we used that $U$ is an AM-space.
Now take $u\in\tu$ arbitrarily then $|u|\in\tu$ as well, and we conclude
\[
|\mathcal{B}_t u|=\bigl|\mathcal{B}_t(u^+-u^-)\bigr|\le\mathcal{B}_tu^++\mathcal{B}_tu^-=\mathcal{B}_t|u|.
\] 
Hence, for every $u\in\tu$ it follows that
\begin{equation*}
\begin{split}
\|\mathcal{B}_t u\|=\bigl\||\mathcal{B}_t u|\bigr\|\le\bigl\|\mathcal{B}_t |u|\bigr\|\le \bigl\lVert A_{-1}^{-1}B\bigr\rVert_{\sL(U,X)}\cdot\bigl\lVert |u|\bigr\rVert_{\infty}=\bigl\lVert A_{-1}^{-1}B\bigr\rVert_{\sL(U,X)}\cdot\lVert u\rVert_{\infty}
\end{split}
\end{equation*}
proving condition (ii). Hence, by \Cref{dens}, $B$ is $\infty$-admissible and the estimate \eqref{bt} holds. Finally, positivity of $\mathcal{B}_t$ follows from \eqref{eq:est-Bt} since the closure of $\tu$ in $\rL^\infty([0,t],U)$ contains $\cou$, cf. \Cref{dens} and its proof.
\end{proof}

The next preliminary result extends the operator $\sBt$ defined in \eqref{bto} from the interval $[0,t]$ to $\mathbb{R}^+$. 
To this end for $u\in\cuor$ we define $\mathcal{B}_0:=0$ and for $t>0$ (with a little abuse of notation) $\mathcal{B}_t u:=\mathcal{B}_t (u|_{[0,t]})$. In this way we can consider $\mathcal{B}_t:\cuor\to X$.
Recall that without loss of generality we still assume that $\omega_0(A)<0$.

\begin{lemma}\label{too}
Let $B:U\to\Xmo$ be positive. Then,
the operator family $(\mathcal{B}_t)_{t\ge0}\subset\sL(\cuor,X)$ is positive, strongly continuous and uniformly bounded of bound $\|A_{-1}^{-1}B\|_{\sL(U,X)}$.
\end{lemma}

\begin{proof} 
By Lemma \ref{binf}, the operator $B$ is $\infty$-admissible for all $t>0$. Moreover, $\mathcal{B}_t$ is positive, and the estimate \eqref{bt} holds for every $t\ge0$ and therefore the family $(\mathcal{B}_t)_{t\ge0}$ uniformly bounded. To show strong continuity we define for $u\in\cuor$ and $0\le r<t$ the translated function $u_{t-r}\in\cuor$ by
\[
u_{t-r}(s):=
\begin{cases}
u(0)&\text{if }0\le s<t-r,\\
u(s-t+r)&\text{if }s\ge t-r.
\end{cases}
\]
Then a simple calculation shows that
\[\mathcal{B}_tu-\mathcal{B}_ru=\mathcal{B}_t(u-u_{t-r})+(T(r)-T(t))A_{-1}^{-1}Bu(0).\]
Clearly, $(T(r)-T(t))A_{-1}^{-1}Bu(0)\to0$ as $t-r\to0$.
Moreover,
\[
\|\mathcal{B}_t(u-u_{t-r})\|_X\le \big\lVert A_{-1}^{-1}B\big\rVert_{\sL(U,X)}\cdot\lVert u-u_{t-r}\rVert_{\infty}.
\]
Since $u\in\cuor$ is uniformly continuous, we conclude
\[
\lim_{t-r\to0}\lVert u-u_{t-r}\rVert_{\infty}=0
\]
which implies strong continuity of the family $(\mathcal{B}_t)_{t\ge0}$.
\end{proof}

 In the sequel we denote by $\tuor$ the space of all the compactly supported $U$-valued step functions on $[0,+\infty)$, i.e., satisfying $\text{supp}(u)\subset[0,b]$  for some $0\le b<+\infty$. In particular, this implies $u(s)=0$ for all $s\ge b$. Moreover,  the closure of $\tuor$ in $\rL^\infty([0,+\infty),U)$ contains $\cuor$.

\begin{lemma}\label{fin}
Let $B:U\to\Xmo$ and $C:X\to U$ be positive operators. Then
\begin{equation*} 
\big(\F u\big)(t):=C\sBt u\quad\text{for }u\in\cuor\text{ and }t\ge0
\end{equation*}
defines a positive operator $\F\in\sL(\cuor)$ satisfying  the estimates
\begin{equation}\label{foon}
\bigl\lVert\F^n\bigr\rVert_{\sL(C_0)}\le\bigl\lVert\left(C A_{-1}^{-1}B\right)^n\bigr\rVert_{\sL(U)}
\quad\text{for every $n\in\mathbb{N}$}.
\end{equation}
\end{lemma}

\begin{proof}
Let $u\in\cuor$. Then by \Cref{too} and the boundedness of $C:X\to U$ it is clear that $\F u:[0,+\infty)\to U$ is a continuous and bounded function. Hence, in any case $\F:\cuor\to C_b([0,+\infty),U)$ is a positive and bounded operator, where $C_b([0,+\infty),U)$ denotes the Banach lattice of all continuous, bounded $U$-valued functions on $[0,+\infty)$ equipped with the sup-norm $\|\cdot\|_\infty$. Now assume in addition that $u$ has compact support contained in the interval $[0,b]$. Since $\omega_0(A)<0$ we obtain for $t>b$
\begin{align*}
\bigl(\F u\bigr)(t)&=CT(t-b)\int_0^b T_{-1}(b-s)Bu(s)\,ds\\
&=CT(t-b)\,\mathcal{B}_b u\to0\quad\text{as }t\to+\infty,
\end{align*}
i.e., $\F u\in\cuor$. Since the functions in $\cuor$ having compact support are dense in $\cuor$ this implies $\rg(\F)\subseteq\cuor$, hence $0\le\F\in\sL(\cuor)$ as claimed. 

\smallbreak
In order to show the estimates \eqref{foon}, we first observe that by \eqref{eq:Bt-sf-cont}  it follows that for every $u\in\tuor$ the function $0\le t\mapsto r(t):=C\sBt u\in U$ is continuous and $r(t)\to0$ as $t\to+\infty$. Hence, $r\in\cuor$ which implies that $\F^n u=\F^{n-1}r$ is well-defined for all $u\in\tuor$ and $n\in\NN$. 

\smallbreak
Next we prove by induction that for any $0\le u\in\tuor$ and $n \in\NN$ it holds
\begin{equation}\label{cabn}
0\le\big(\F^n u\big)(t)\le\bigl(-C A_{-1}^{-1}B\bigr)^n\bar{u}\quad\text{for all $t\ge0$},
\end{equation}
where for $u(s):=\sum_{n=1}^{N}u_n\,\one_{I_n}(s)$, $s\in[0,+\infty)$, with pairwise disjoint intervals $I_n$,  we define as above $\ub=\sup_{n=1,\dots,N}u_n\in U_+$. 

For $n=1$ this estimate follows immediately by multiplying \eqref{eq:est-Bt} from the left by $C\ge0$.  
Now assume that \eqref{cabn} holds for some fixed $n\ge1$. Then using that $-A_{-1}^{-1}B\ge0$ we obtain
\begin{equation*}
\begin{split}
0\le\big(\F^{n+1} u\big)(t)
&=C\mathcal{B}_t\big(\F^n u\big)\le C\mathcal{B}_t\bigl(-C A_{-1}^{-1}B\bigr)^n\bar{u}\\
&= C\int_{0}^{t}T_{-1}(s)B\bigl(-C A_{-1}^{-1}B\bigr)^n\bar{u}\,ds\\
&=-C A_{-1}^{-1}B\bigl(-C A_{-1}^{-1}B\bigr)^n\bar{u}-C T(t)\bigl(-A_{-1}^{-1}B\bigr)\bigl(-C A_{-1}^{-1}B\bigr)^n\bar{u}\\
&\le\bigl(-C A_{-1}^{-1}B\bigr)^{n+1}\bar{u},
\end{split}
\end{equation*}
which proves \eqref{cabn} by induction. 
Now take an arbitrary $u\in\tuor$. Then $|u|\in\tuor$ as well, and we conclude
\[
|\F^n u|=\bigl|\F^n(u^+-u^-)\bigr|\le\F^n u^++\F^n u^-=\F^n|u|,
\]
i.e.,  $|\F^n u|(t)\le (\F^n|u|)(t)$ for all $t\ge0$.
Using \eqref{cabn} it follows that
\begin{equation}\label{eq:est-Fin}
\begin{split}
\bigl\|(\F^n u)(t)\bigr\|_U&=\bigl\||\F^n u|(t)\bigr\|_U\le\bigl\|\bigl(\F^n |u|\bigr)(t)\bigr\|_U\\
&\le \bigl\lVert \bigl(CA_{-1}^{-1}B\bigr)^n\bigr\rVert_{\sL(U)}\cdot\lVert \ub\rVert_U
=\bigl\lVert \bigl(CA_{-1}^{-1}B\bigr)^n\bigr\rVert_{\sL(U)}\cdot\lVert u\rVert_{\infty},
\end{split}
\end{equation}
where in the last equality we used the AM-property of $U$. This shows that the operator 
$\F^n:\tuor\subset \rL^\infty([0,+\infty),U)\to\cuor$ is bounded of bound $\lVert(CA_{-1}^{-1}B)^n\rVert_{\sL(U)}$. Since $\cuor$ is complete, $\F^n$ has a unique bounded extension $\sR_n=\F^{n-1}\cdot\sR_1$ to the closure of $\tuor$ in $\rL^\infty([0,+\infty),U)$ having the same bound $\lVert(CA_{-1}^{-1}B)^n\rVert_{\sL(U)}$.  
Now as in the proof of \Cref{dens} it follows that $\sR_1|_{\cuor}=\F$, hence $\sR_n|_{\cuor}=\F^n$ which implies \eqref{foon}.
\end{proof}

Our last preliminary result deals with the invertibility of $Id-\mathcal{F}_\infty$ and the Laplace transform of its inverse. The statement heavily depends on the spectral assumption $\sr(CR(\lambda,A_{-1})B)<1$. In the sequel $\mathscr{L}(\cdot)$ denotes the Laplace transform as introduced in \cite[Sect.1.4]{ABHN:11}. Recall that, as always, we assume $\omega_0(A)<0$. 

\begin{lemma}\label{ltam}
Let $B:U\to\Xmo$ and $C:X\to U$ be positive and
assume that $\sr(CA_{-1}^{-1}B)<1$. Then for any $\lambda>0$ one has $1\in\rho(\F)\cap\rho(CR(\lambda,A_{-1})B)$ and
\begin{equation}\label{L}
\mathscr{L}\big((Id-\F)^{-1}u(\cdot)\big)(\lambda)=\big(Id-CR(\lambda,A_{-1})B\big)^{-1}\cdot\hat{u}(\lambda)
\end{equation}
for all $u\in\cuor$.
\end{lemma}
\begin{proof}	
We start by showing that $1\in\rho(CR(\lambda,A_{-1})B)$ for all $\lambda>0$.
By \Cref{L:B-pos} we have $R(\lambda,A_{-1})B\ge0$ for any $\lambda>\spb(A)$. Since $\spb(A)\le\omega_0(A)<0$, the resolvent identity implies
\[0\le CR(\lambda,A_{-1})B=CR(0,A_{-1})B-C\lambda R(\lambda,A)R(0,A_{-1})B\le -CA_{-1}^{-1}B\] 
for all $\lambda>0$. 
Hence, by Proposition \ref{spr} we conclude that for all $\lambda>0$
\begin{equation}\label{spes}
\sr\bigl(C R(\lambda,A_{-1})B\bigr)\le \sr\bigl(-C A_{-1}^{-1}B\bigr)=\sr\bigl(C A_{-1}^{-1}B\bigr)<1,
\end{equation}
In particular this implies $1\in\rho(CR(\lambda,A_{-1})B)$ for all $\lambda>0$ as claimed. 

\smallbreak
Next observe that by \eqref{foon} we have
\begin{equation}\label{eq:est-rF}
\sr(\F)=\lim_{n\to+\infty}\bigl\|\F^n\bigr\|^{1/n}_{\sL(C_0)}\le\lim_{n\to+\infty}\bigl\|(CA_{-1}^{-1}B)^n\bigr\|^{1/n}_{\sL(U)}=\sr(CA_{-1}^{-1}B)<1.
\end{equation}
Hence, $Id-\F$ is invertible and its inverse is given by the Neumann series
\begin{equation}\label{eq:Neu-F_infty}
(Id-\F)^{-1}=\sum_{n=0}^{\infty}\F^n
\end{equation}
which converges in  $\sL(\cuor)$.
It only remains to show \eqref{L}.

To this end we first verify that for all $u\in\cuor$ and $n\in\mathbb{N}$ we have
\begin{equation}\label{lap1n}
\mathscr{L}(\F^n u)(\lambda)=\big(C\,R(\lambda,A_{-1})\,B\big)^n\cdot \hat{u}(\lambda),\quad\lambda>0.
\end{equation}
Indeed, for $n=1$, recalling that $C$ is bounded, for $\lambda>0$ and $u\in\cuor$ we obtain
\begin{equation*}
\begin{split}
\mathscr{L}(\mathcal{F}_\infty u)(\lambda)
&=\int_{0}^{+\infty}e^{-\lambda t}(\mathcal{F}_\infty u)(t)\,dt
\\
&=C\int_{0}^{+\infty}e^{-\lambda t}\left(\int_{0}^{t}T_{-1}(t-s)Bu(s)\,ds\right)\,dt\\
&=C\int_{0}^{+\infty}e^{-\lambda t}\big(T_{-1}(\cdot)B\star u(\cdot)\big)(t)\,dt\\
&=C\cdot\sL\big(T_{-1}(\cdot)B\star u(\cdot)\big)(\lambda)
=C R(\lambda,A_{-1})B\cdot\hat{u}(\lambda),
\end{split}
\end{equation*}
where in the last line we have used the convolution theorem for Laplace transform (see, e.g., \cite[Lemma 3.12]{BE:14}). 
If \eqref{lap1n} holds for some $n\ge1$, then for $u\in\cuor$ and $\lambda>0$ we have
\begin{equation*}
\begin{split}
\mathscr{L}\bigl(\F^{n+1}u\bigr)(\lambda)
&=\int_{0}^{+\infty}e^{-\lambda t}\Big(C\int_{0}^{t}T_{-1}(t-s)B(\F^nu)(s)\,ds\Big)\,dt\\
&=C\int_{0}^{+\infty}e^{-\lambda t}\Big(T_{-1}(\cdot)B\star\F^nu(\cdot)\Big)(t)\,dt\\
&=C\cdot\sL\Big(T_{-1}(\cdot)B\star\F^nu(\cdot)\Big)(\lambda)\\
&=CR(\lambda,A_{-1})B\cdot\sL(\F^nu)(\lambda)
=\big(C\,R(\lambda,A_{-1})\,B\big)^{n+1}\cdot \hat{u}(\lambda),
\end{split}
\end{equation*}
which implies \eqref{lap1n} by induction.
Finally, we use \eqref{lap1n} to prove that \eqref{L} holds for all $u\in\cuor$ and $\lambda>0$. To his aim,  we define the sequence of continuous functions $f_n\in\cuor$, $n\in\mathbb{N}$ by
\begin{equation*} 
f_n:=\sum_{k=0}^{n}\F^ku(\cdot).
\end{equation*}
Using \eqref{lap1n} we immediately obtain that
\[\hat{f}_n(\lambda)=\sum_{k=0}^{n}\bigl(CR(\lambda,A_{-1})B\bigr)^k\cdot\hat{u}(\lambda)
\quad\text{ for all }\lambda>0.
\]
Moreover, norm convergence of the Neumann series \eqref{eq:Neu-F_infty} yields the uniform convergence of $(f_n)_{n\in\NN}$ to
$f:=(Id-\F)^{-1}u(\cdot)$.
By \cite[Thm.1.7.5]{ABHN:11}  this implies
\begin{equation*}
\begin{split}
&\mathscr{L}\big((Id-\F)^{-1}u(\cdot)\big)(\lambda)=\hat{f}(\lambda)=\lim_{n\to+\infty}\hat{f}_n(\lambda)\\
&=\lim_{n\to+\infty}\sum_{k=0}^{n}\bigl(CR(\lambda,A_{-1})B\bigr)^k\cdot\hat{u}(\lambda)
=\big(Id-CR(\lambda,A_{-1})B\big)^{-1}\cdot\hat{u}(\lambda)
\end{split}
\end{equation*}
for all $u\in\cuor$ and $\lambda>0$, i.e., \eqref{L}.
\end{proof}

We have now all the necessary tools to prove the main result of this section.
\begin{proof}[Proof of \Cref{AM}]
We define an operator family $(S(t))_{t\ge0}\subset\sL(X)$ and prove that it is a  $C_0$-semigroups with generator $A_{BC}$ given in \eqref{abcam}.

\smallskip
To this end we first introduce the positive bounded linear operator $\mathcal{C}_\infty:X\to\cuor$ by $(\mathcal{C}_\infty x)(t):=CT(t)x$.
Then we use the invertibility of $Id-\mathcal{F}_\infty$ proved in \Cref{ltam}  to define
\begin{equation}\label{semig}
S(t):=T(t)+\mathcal{B}_t(Id-\mathcal{F}_\infty)^{-1}\mathcal{C}_\infty\in\sL(X)
\quad\text{for $t\ge0$.}
\end{equation}
Now by \Cref{too}, the operator family $(S(t))_{t\ge0}$ is strongly continuous on $X$ and uniformly bounded. 
The same computations as in \cite[Proof of Thm.10]{ABE:14} then show that
\begin{equation*}
\mathscr{L}\bigl(S(\cdot)x\bigr)(\lambda)=R(\lambda,A_{BC})x\quad\text{for all $\lambda>0$ and $x\in X$}.
\end{equation*}
By \cite[Thm.3.1.7]{ABHN:11}, this implies that the operator family \eqref{semig} is a positive  $C_0$-semigroup on $X$ with generator $A_{BC}$ defined in \eqref{abcam}. Finally, the variation of parameters formula \eqref{VP1} follows as in \cite[Thm.10]{ABE:14}.
\end{proof}

\subsection{Positive Structured Perturbations factorized via AL-Spaces}\label{pMV}
The following result generalizes a perturbation theorem for positive semigroups due to Desch \cite{Des:88} and Voigt \cite{Voi:89}, see \Cref{rem:DV-pos}. With respect to our setting from \Cref{GS} we assume throughout this subsection that $\rg(B)\subseteq X$ which implies that $B\in\sL(U,X)$ is bounded.

\begin{theorem}\label{AL}
Let $A:D(A)\subseteq X\to X$ be the generator of a positive  $C_0$-semigroup $(T(t))_{t\ge0}$ on a Banach lattice $X$. Assume $B:U\to X$ and $C:D(A)\to U$ to be positive linear operators, where $U$ is an AL-space.
If the spectral radius $\sr(CR(\lambda,A)B)<1$ for some $\lambda>\omega_0(A)$, then the operator
\begin{equation}\label{abcal}
A_{BC}:=A+BC\quad\text{with domain}\,\,D(A_{BC})=D(A)
\end{equation}
generates a positive  $C_0$-semigroup $(S(t))_{t\ge0}$ on $X$ satisfying the variation of parameters formula
\begin{equation}\label{VP2}
S(t)x:=T(t)x+\int_{0}^{t}T(t-s)\cdot BC\cdot S(s)x\,ds,\quad x\in D(A),\,\,t\ge0.
\end{equation}
\end{theorem}

\begin{remark}\label{rem:DV-pos}
As already mentioned, this result generalizes the main result in \cite{Des:88} and \cite{Voi:89}, see also \cite[Sect.13.3]{BKR:17} or \cite[Sect.5.2.1]{BA:06}, where it is assumed that $X=U$ and $B=Id$. In particular, in these works  one needs $X$ to be an AL-space while in \Cref{AL} only the boundary space $U$ has to be of the type AL while $X$ is allowed to be an arbitrary Banach lattice. See \Cref{subsec:uP-FD} for an example where this generalization is crucial.
\end{remark}

The proof of \Cref{AL} is structured similarly as the one of \Cref{AM}. We note that by the rescaling argument from \cite[Sect.II.2.2]{EN:00}  we  again assume $\omega_0(A)<0$ and choose $\lambda=0$ throughout the proof.

\smallbreak
Under the assumptions of \Cref{AL} the control operator $B:U\to X$ is bounded and hence $1$-admissible for all $t>0$, cf.\ \Cref{badm}. Next we show that the observation operator $C:D(A)\to U$ is $1$-admissible as well.

\begin{lemma}\label{c1}
If $C:D(A)\to U$ is a positive linear operator and $U$ is an AL-space, then $C$ is 1-admissible for any $t>0$ and
$\mathcal{C}_t:X\to\lru$ defined by \eqref{cto} satisfies
\begin{equation}\label{ct2}
\lVert\mathcal{C}_t\rVert_{\sL(X,\rL^1)}\le\bigl\lVert CA^{-1}\bigr\rVert_{\sL(X,U)}\quad\text{for all $t>0$}.
\end{equation}
\end{lemma} 

\begin{proof} We follow in part the proof of \cite[Prop.13.7]{BKR:17}.
First take $0\le x\in D(A)$. Since $U$ is an AL-space and $ CT(s)x\ge0$ for all $s\ge0$ we conclude
\begin{equation}\label{cpos}
\begin{split}
\int_{0}^{t}\bigl\lVert CT(s)x\bigr\rVert_U\,ds
&=\left\lVert \int_{0}^{t}C T(s)x\,ds\right\rVert_U
\le\left\lVert \int_{0}^{+\infty}CT(s)x\,ds\right\rVert_U\\
&=\bigl\lVert-CA^{-1}x\bigr\rVert_U
\le\bigl\lVert CA^{-1}\bigr\rVert_{\sL(X,U)}\cdot\lVert x\rVert_X.
\end{split}
\end{equation}
If $x\in D(A)$ we decompose it as $x=x^+-x^-$ and define for $\lambda>0$
\begin{align*}
x_\lambda&:=\lambda R(\lambda,A)x\in D(A),\\
x_\lambda^P&:=\lambda R(\lambda,A)x^+\in D(A)_+,\\
x_\lambda^N&:=\lambda R(\lambda,A)x^-\in D(A)_+.
\end{align*}
Then $x_\lambda\to x$, $x_\lambda^P\to x^+$ and $x_\lambda^N\to x^-$ in $X$ as $\lambda\to+\infty$. Moreover,
if $\lVert\cdot\rVert_{X_1}$ denotes the norm of $X_1=D(A)$ given by $\|x\|_{X_1}=\|Ax\|_X$, then also
\begin{equation*}
\begin{split}
\lVert x_\lambda-x\rVert_{X_1}=\bigl\lVert \lambda R(\lambda,A)x-x\bigr\rVert_{X_1}=\bigl\lVert \lambda R(\lambda,A)Ax-Ax\bigr\rVert_X\to0,
\end{split}
\end{equation*}
i.e., $x_\lambda\to x$ in $X_1$ as $\lambda\to+\infty$.
Since for all $s\in[0,t]$ we have
\[\bigl\lvert CT(s)x_\lambda\bigr\rvert\le CT(s)\big(x_\lambda^P+x_\lambda^N\big)\]
using \eqref{cpos} we obtain
\begin{equation} \label{cla}
\begin{split}
\int_{0}^{t}\bigl\lVert CT(s)x_\lambda\bigr\rVert_U\,ds
&\le\int_{0}^{t}\left\lVert CT(s)\big(x_\lambda^P+x_\lambda^N\big)\right\rVert_U\,ds\\
&\le\bigl\lVert CA^{-1}\bigr\rVert_{\sL(X,U)}\cdot\bigl\lVert x_\lambda^P+ x_\lambda^N\bigr\rVert_X.
\end{split}
\end{equation}
Now choose $M\ge1$ such that $\|T(s)\|\le M$ for all $s\ge0$. Then
\begin{align*}
\bigl\|CT(s)x_\lambda-CT(s)x\bigr\|_U&=\bigl\|CA^{-1}T(s)(Ax_\lambda-Ax)\bigr\|_U\\
&\le M\cdot\bigl\|CA^{-1}\bigr\|_{\sL(X,U)}\cdot\|x_\lambda-x\|_{X_1}.
\end{align*}
Hence, $CT(s)x_\lambda\to CT(s)x$ uniformly for $s\in[0,t]$ as $\lambda\to+\infty$ and \eqref{cla} implies
\begin{equation*}
\int_{0}^{t}\bigl\lVert CT(s)x\bigr\rVert_U\,ds
\le\bigl\lVert CA^{-1}\bigr\rVert_{\sL(X,U)}\cdot\lVert x^++ x^-\rVert_X
=\bigl\lVert CA^{-1}\bigr\rVert_{\sL(X,U)}\cdot\lVert x\rVert_X.
\end{equation*}
This proves that $C$ is 1-admissible for all $t>0$ and that the operator $\mathcal{C}_t$ defined in \eqref{cto} satisfies the estimate \eqref{ct2}.
\end{proof}

As in \Cref{pDS}, we need to extend the operators defined in \eqref{btpo} and \eqref{cto} on $\mathbb{R}^+$. 
As always, we assume that $\omega_0(A)<0$.

\begin{lemma}\label{too2}
Let $B:U\to X$ and $C:D(A)\to U$ be positive operators. Then,
\begin{enumerate}[label=(\roman*)]
\item $B$ is 1-admissible for every $t>0$, i.e., $\mathcal{B}_t\in\sL(\lru,X)$ for every $t>0$ where
\begin{equation}\label{b}
\mathcal{B}_t u:=\int_{0}^{t}T(t-s)Bu(s)\,ds.\quad u\in\lru.
\end{equation}
Moreover, the family $(\mathcal{B}_t)_{t\ge0}$ is strongly continuous and uniformly bounded.
\item $C$ is 1-admissible for all $t\ge0$, i.e., there exists a bounded operator $\mathcal{C}_\infty:X\to\lru$ such that
\begin{equation*} 
\big(\mathcal{C}_\infty x\big)(s)=CT(s)x\quad\text{for}\,\,x\in D(A),\,\,s\ge0.
\end{equation*}
\end{enumerate}
\begin{proof}
(i).  The $1$-admissibility of $B$ for every $t>0$ and strong continuity of $(\mathcal{B}_t)_{t\ge0}$ follow easily from the boundedness of $B:U\to X$. Uniform boundedness of $(\mathcal{B}_t)_{t\ge0}$ follows from the standing assumption $\omega_0(A)<0$.
(ii) is clear by \eqref{ct2}. 
\end{proof}
\end{lemma}

Next we  use again the space of all compactly supported $U$-valued step functions in $[0,+\infty)$, denoted as before by $\tuor$ which forms a dense sublattice of $\lru$.

\begin{lemma}\label{fin2}
Let $B:U\to X$ and $C:D(A)\to U$ be positive operators. Then
\begin{equation}\label{ftoal}
\left(\F u\right)(t):=C\int_{0}^{t}T(t-s)Bu(s)\,ds,\quad t\ge0,
\end{equation}
is well-defined for all $u\in\tuor$. Moreover, it possesses a (unique) positive bounded extension (still denoted by $\F$) $\F\in\sL(\lru)$ satisfying the estimates
\begin{equation}\label{foon2}
\bigl\lVert\F^n\bigr\rVert_{\sL(\rL^1)}\le\bigl\lVert (CA^{-1}B)^n\bigr\rVert_{\sL(U)}
\quad\text{for every }n\in\NN.
\end{equation}
\end{lemma}

\begin{proof}
We verify that for all $u\in\tuor$ we have $\F u\in\lru$ and that the following two conditions hold:
\begin{equation*} 
\begin{aligned}
	&(i)\quad \int_{0}^{t}T(t-s)Bu(s)\,ds\in D(A)\text{ for any $t>0$,}\\
	&(ii)\quad \lVert\F u\rVert_1\le \bigl\lVert CA^{-1}B\bigr\rVert_{\sL(U)}\cdot\lVert u\rVert_1.
\end{aligned}
\end{equation*}
Condition (i) follows by the same reasoning as in the proof of \Cref{binf}, bearing in mind that for every $n=1,\ldots,N$ the term $Bu_n$ in \eqref{eq:Bt-sf-cont} now belongs to $X$, hence the right-hand-side belongs to $D(A)$. Moreover, from \eqref{ftoal} it immediately follows that $\F$ is positive.

\smallbreak
To prove condition (ii), fix $u\in\tuor$. Then $|u|\in\tuor$ as well and since $\F$ is positive we conclude
\[\bigl\lvert(\F u)(t)\bigr\rvert
\le(\F u^+)(t)+(\F u^-)(t)
=\bigl(\F|u|\bigr)(t)\quad\text{for all}\,\,t\ge0.\]
Since $U$ is an AL-space and using Fubini's Theorem and (i), we obtain
\begin{equation}\label{eq:Fub-AL}
\begin{split}
\lVert\mathcal{F}_\infty u\rVert_{1}
&\le\int_{0}^{+\infty}\bigl\lVert\bigl(\mathcal{F}_\infty |u|\bigr)(t)\bigr\rVert_U\,dt\\
 &=\int_{0}^{+\infty}\left\lVert C\int_{0}^{t}T(t-s)B|u(s)|\,ds\right\rVert_Udt\\
 &=\left\lVert\int_{0}^{+\infty}CA^{-1}\int_{0}^{t}A_{-1}T(t-s)B|u(s)|\,ds\,dt\right\rVert_U\\
 &=\left\lVert CA^{-1}\int_{0}^{+\infty}\left(\int_{s}^{+\infty}A_{-1}T(t-s)B|u(s)|\,dt\right)\,ds\right\rVert_U\\
 &=\left\lVert CA^{-1}\int_{0}^{+\infty}\left(\int_{0}^{+\infty}A_{-1}T(r)B|u(s)|\,dr\right)\,ds\right\rVert_U\\
 &=\left\lVert-C A^{-1}B\int_{0}^{+\infty}|u(s)|\,ds\right\rVert_U\le\bigl\lVert C A^{-1}B\bigr\rVert_{\sL(U)}\cdot\lVert u\rVert_{1}.
\end{split}
\end{equation}
Therefore,  $(\F u)(t)$ given by \eqref{ftoal} defines a positive linear operator 
\[\mathcal{F}_\infty:\tuor\subset\lru\to\lru\] 
satisfying condition (ii) for all $u\in\tuor$. Hence, $\F$ possesses a unique bounded positive extension $\F\in\sL(\lru)$ satisfying the norm estimate \eqref{foon2} for $n=1$.

\smallbreak
To prove \eqref{foon2} for arbitrary $n\in\NN$ we first verify by induction that for any $u\in\lru$ and $t\ge0$ we have
\begin{equation}\label{cabn2}
\int_{0}^{+\infty}\bigl\lvert(\F^nu)(t)\bigr\rvert\,dt\le\bigl(-CA^{-1}B\bigr)^n\int_{0}^{+\infty}\bigl|u(t)\bigr|\,dt\quad\text{for all}\,\,n\in\mathbb{N}.
\end{equation}
Using Fubini's Theorem like in  \eqref{eq:Fub-AL} and the positivity of $\F$ the case $n=1$ follows, since
\begin{equation*}
\begin{split}
\int_{0}^{+\infty}\bigl\lvert(\F u)(t)\bigr\rvert\,dt\le\int_{0}^{+\infty}\bigl(\F|u|\bigr)(t)\,dt=-CA^{-1}B\int_{0}^{+\infty}|u(t)|\,dt.
\end{split}
\end{equation*}
Now assume \eqref{cabn2} holds for some $n\ge 1$. Then, again by the positivity of $\F^n$ we obtain
\begin{equation}\label{eq:Fin-L1}
\begin{split}
\int_{0}^{+\infty}\bigl\lvert(\F^{n+1}u)(t)\bigr\rvert\,dt
&=\int_{0}^{+\infty}\bigl\lvert\bigl(\F^{n}(\F u)\bigr)(t)\bigr\rvert\,dt\\
&\le\int_{0}^{+\infty}\bigl(\F^{n}\left\lvert\F u\right\rvert\bigr)(t)\,dt\\
&\le\bigl(-CA^{-1}B\bigr)^n\int_{0}^{+\infty}\bigl|(\F u)(t)\bigr|\,dt\\
&\le\bigl(-CA^{-1}B\bigr)^{n+1}\int_{0}^{+\infty}\bigl|u(t)\bigr|\,dt,
\end{split}
\end{equation}
which proves \eqref{cabn2}. 
Keeping in mind that $U$ is an AL-space, \eqref{cabn2}  implies for arbitrary $u\in\lru$, $t\ge0$ and $n\in\mathbb{N}$
\begin{equation}\label{eq:Fin-L1-n}
\begin{split}
\bigl\|\F^nu\bigr\|_1&=
\int_{0}^{+\infty}\bigl\lVert(\F^n u)(t)\bigr\rVert_U\,dt\\
&=\left\lVert\int_{0}^{+\infty}\bigl\lvert(\F^n u)(t)\bigr\rvert\,dt\right\rVert_U\\
&\le\left\lVert\bigl(-CA^{-1}B\bigr)^n\int_{0}^{+\infty}\bigl|u(t)\bigr|\,dt\right\rVert_U
\le\Bigl\lVert\bigl(CA^{-1}B\bigr)^n\Bigr\rVert_{\sL(U)}\cdot\lVert u\rVert_1.
\end{split}
\end{equation}
This completes the proof of \eqref{foon2}.
\end{proof}

At this point we need a result similar to \Cref{ltam} concerning the invertibility of $Id-\mathcal{F}_\infty$ and the Laplace transform of its inverse. Again, for this the spectral assumption $\sr(CR(\lambda,A)B)<1$ plays a crucial role. Recall that, as always, we assume $\omega_0(A)<0$. 

\begin{lemma}\label{ltal}
Let $B:U\to X$ and $C:D(A)\to U$ be positive
and assume that $\sr(CA^{-1}B)<1$. Then for any $\lambda>0$ one has $1\in\rho(\F)\cap\rho(CR(\lambda,A)B)$ and
\begin{equation}\label{lapl}
\mathscr{L}\big((Id-\mathcal{F}_\infty)^{-1}u\big)(\lambda)=\big(Id-CR(\lambda,A)B\big)^{-1}\cdot\hat{u}(\lambda),
\end{equation}
for all $u\in\lru$.
\end{lemma}
\begin{proof}
The claim $1\in\rho(CR(\lambda,A)B)$ follows like in the proof of \Cref{ltam} from \eqref{spes}.  Moreover, \eqref{foon2} implies $\sr(\F)\le \sr(CA^{-1}B)<1$ as in \eqref{eq:est-rF} and hence $1\in\rho(\F)$.
The identity \eqref{lapl} then follows by \cite[Lem.13]{ABE:14}.
\end{proof}

We have now all the necessary tools to prove the main result of this section.
\begin{proof}[Proof of \Cref{AL}]
We define an operator family $(S(t))_{t\ge0}\subset\sL(X)$ and verify that it is a  $C_0$-semigroups with generator $A_{BC}$.

\smallbreak
Since by \Cref{ltal}, $Id-\mathcal{F}_\infty$ is invertible we can define the operators
\begin{equation}\label{semig2}
S(t):=T(t)+\mathcal{B}_t(Id-\mathcal{F}_\infty)^{-1}\mathcal{C}_\infty\in\sL(X)
\quad\text{for $t\ge0$},
\end{equation}
where $\mathcal{B}_t$ is given by \eqref{b}. Then \Cref{too2} implies that $(S(t))_{t\ge0}$ is strongly continuous on $X$ and uniformly bounded. 
The same reasoning as in \cite[Proof of Thm.10]{ABE:14} gives
\begin{equation*}
\mathscr{L}\bigl(S(\cdot)x\bigr)(\lambda)=R(\lambda,A_{BC})x\quad\text{for all $x\in X$ and $\lambda>0$}.
\end{equation*}
By \cite[Thm.3.1.7]{ABHN:11}, this implies that \eqref{semig2} defines a positive  $C_0$-semigroup on $X$ with generator $A_{BC}$ given by \eqref{abcal}. Finally, the variation of parameters formula \eqref{VP2} follows as in \cite[Thm.10]{ABE:14}. 
\end{proof}

\subsection{Positive Structured Perturbations factorized via $\mathbf{\R}$}\label{gen}

In this subsection we assume the space $U$ to be both an AM- and an AL-space. Since no infinite dimensional space can satisfy these properties at the same time, we choose in the sequel $U=\R$ for some $N\ge1$. Then $U$ becomes an AL-space when equipped with the 1-norm $\lvert\cdot\rvert_1$, whereas the $\infty$-norm $\lvert\cdot\rvert_\infty$ makes it an AM-space. In some of the following the proofs we need to make use of the AM- and/or AL-property,  hence use the notations $\V:=(\R,\lvert\cdot\rvert_1)$ and $\W:=(\R,\lvert\cdot\rvert_\infty)$ when the norm needs to be specified.

\smallbreak
The following is the main result of this section. In contrast to \Cref{AM,AL} it allows both operators $B:U\to\Xmo$ and $C:Z\to U$ to be unbounded\footnote{Here by ``both unbounded'' we intend that $\rg(B)\not\subset X$ and $D(C)=Z\subsetneqq X$, respectively.}%
.  Compared to the Weiss--Staffans perturbation theorem \cite[Thm.10]{ABE:14} it replaces the in general difficult to verify boundedness and spectral conditions on $\F$ by the much simpler condition $\sr(CR(\lambda,A_{-1})B)<1$. Note that in case $U=\R$ under the usual positivity assumptions the operator $CR(\lambda,A_{-1})B$ gets a positive scalar $N\times N$-matrix.

\begin{theorem}\label{PT}
Let $A:D(A)\subseteq X\to X$ generate a positive  $C_0$-semigroup $(T(t))_{t\geq0}$ on the Banach lattice $X$. Assume that $B:\R\to\Xmo$ and $C:Z\to\R$ are positive linear operators where $D(A)\subseteq Z\subseteq X$. Moreover, assume that the triple $(A,B,C)$ is compatible and the spectral radius $\sr(C R(\lambda,A_{-1})B)<1$ for some $\lambda>\omega_0(A)$.
If there exists $1\le p<\infty$ such that $B$ and $C$ are $p$-admissible, then the operator
\begin{equation}\label{abc}
	\begin{aligned}
		\hspace{3cm} A_{BC}&:=(A_{-1}+BC)|_{X},\\
		D(A_{BC})&:=\bigl\{x\in Z:(A_{-1}+BC)x\in X\bigr\}
	\end{aligned}
\end{equation}
generates a positive  $C_0$-semigroup $(S(t))_{t\geq0}$ on $X$ satisfying the \emph{variation of parameters formula}
\begin{equation}\label{VP}
S(t)x:=T(t)x+\int_{0}^{t}T_{-1}(t-s)\cdot BC\cdot S(s)x\,ds\quad\text{for}\,\,x\in D(A_{BC}),\,\,t\ge0.
\end{equation}		
\end{theorem}

For an application of this result we refer to \Cref{subsec:r1p-FD}.

\smallbreak
As in \Cref{pDS,pMV}, in the sequel we may assume that $\omega_0(A)<0$ and take $\lambda=0$.
We start by recalling that under this assumption the operators $\mathcal{B}_t$ and $\mathcal{C}_t$ defined in \eqref{btpo} and \eqref{cto} can be extended to $\mathbb{R}^+$ in the following manner.

\begin{lemma}\label{toop}
Let $(A,B,C)$ be a compatible triple and $B:\R\to\Xmo$ and $C:Z\to\R$ be $p$-admissible positive operators. Then
\begin{enumerate}[label=(\roman*)]
\item $B$ is $p$-admissible for all $t\ge0$, i.e.,  $\mathcal{B}_t\in\sL(\lprr,X)$ for every $t\ge0$. In addition, the family $(\mathcal{B}_t)_{t\ge0}$ is strongly continuous and uniformly bounded;
\item $C$ is $p$-admissible for all $t\ge0$, i.e., there exists $\mathcal{C}_\infty\in\sL(X,\lprr)$ satisfying
\begin{equation*}
	\big(\mathcal{C}_\infty x\big)(t):=CT(t)x\quad\text{for all $x\in D(A)$ and $t\ge0$.}
\end{equation*}
\end{enumerate}
\end{lemma}
\begin{proof} 
See \cite[Lem.11]{ABE:14}. 
\end{proof}

In the sequel we denote by $\ceor:=$
\[
\Biggl\{f\in C\bigl((0,+\infty),\R\bigr)\ \bigg|\ 
\begin{aligned}
&\text{$\forall\varepsilon>0$  $\exists$  compact $K_\varepsilon\subset(0,+\infty)$}\\
&\text{such that $|f(s)|_{\R}<\varepsilon$ $\forall s\in(0,+\infty)\setminus K_\varepsilon$}
\end{aligned}
\Biggr\}
\]
the Banach space of all continuous $\R$-valued functions vanishing at zero and infinity endowed with the $\infty$-norm given by
\[
\|f\|_\infty:=\sup_{s\ge0}\bigl|f(s)\bigr|_{\R}.
\]
Here $\lvert\cdot\rvert_{\R}$ denotes an arbitrary norm on $\R$ where different choices of $|\cdot|_{\R}$ yield to equivalent norms.
Then we define the auxiliary space
\begin{equation*}
\mathcal{D}:= C_c^2\bigl((0,+\infty),\R\bigr)
\end{equation*}
which is a dense subspace of $\ceor$ as well as of $\lprr$ for all $1\le p<+\infty$.
Since for every $u\in\mathcal{D}$ there exist $0<a<b<+\infty$ such that $\text{supp}(u)\subset[a,b]$ it follows $u^{(k)}(0)=0$ for $k=0,1,2$.
The subset of $\mathcal{D}$ containing the positive functions will be denoted by $\mathcal{D}_+$. If $u\in\ceor$ or $u\in\lprr$, then thanks to the lattice setting we can decompose it as $u=u^+-u^-$, where $u^+$ and $u^-$, respectively, denote the positive and the negative part of $u$.
Finally, we define
\begin{equation}\label{hatu}
\bar{u}:=\lVert u\rVert_\infty\cdot 1_{\W}\in \R,
\end{equation}
where $1_{\W}:=(1,\dots,1)\in\R$ is the unit of $\R$. This element can also be identified by the constant function $t\mapsto\bar{u}$, which belongs to the Banach lattice $C_b([0,+\infty),\R)$ of all bounded continuous $\R$-valued functions on the interval $[0,+\infty)$ equipped with the sup-norm $\|\cdot\|_\infty$. Moreover, it holds
\begin{equation*}
|u|\le\bar{u}\quad\text{and}\quad\lVert\bar{u}\rVert_\infty=\lVert u\rVert_\infty.
\end{equation*}

With this notation we have the following result
\begin{lemma}\label{lemmaoo}
Let $(A,B,C)$ be compatible and $B:\R\to\Xmo$ and $C:Z\to\R$ be positive. Then
\begin{equation}\label{fto}
\big(\mathcal{F}_\infty u\big)(t):=C\int_{0}^{t}T_{-1}(t-s)Bu(s)\,ds,\quad t\ge0
\end{equation}
is well-defined for all $u\in\mathcal{D}\cup|\mathcal{D}|$ where $|\mathcal{D}|:=\{|u|:u\in\mathcal{D}\}$. In addition, 
\begin{enumerate}[label=(\roman*)]
\item $\mathcal{F}_\infty$ possesses a unique bounded positive extension (still denoted by $\mathcal{F}_\infty$) $\F\in\sL(\cwor)$ satisfying
\begin{equation} \label{eq:Fi-C0}
\bigl\lVert\mathcal{F}_\infty^n\bigr\rVert_{\sL(C_0)}\le\bigl\lVert (CA_{-1}^{-1}B)^n\bigr\rVert_{\sL(\W)}\quad\text{for all $n\in\NN$};
\end{equation}
\item $\mathcal{F}_\infty$ possesses a unique bounded positive extension (still denoted by $\mathcal{F}_\infty$) $\mathcal{F}_\infty\in\sL(\lrv)$ satisfying
\begin{equation} \label{eq:Fi-L1}
\bigl\lVert\mathcal{F}_\infty^n\bigr\rVert_{\sL(\rL^1)}\le\bigl\lVert (CA_{-1}^{-1}B)^n\bigr\rVert_{\sL(\V)}\quad\text{for all $n\in\NN$}.
\end{equation}
\end{enumerate}
\end{lemma}

\begin{proof}
We first prove that $(\mathcal{F}_{\infty}u)(t)$ in \eqref{fto} is well-defined for $u\in\mathcal{D}$.
Let $u\in\mathcal{D}$ and $t\ge0$. Integrating by parts twice and using that $u(0)=u'(0)=0$, we conclude
\begin{align}\label{daz}
\int_{0}^{t}&T_{-1}(t-s)Bu(s)ds\notag\\
&=\big[-A_{-1}^{-1}T_{-1}(t-s)Bu(s)\big]_{0}^{t}+\int_{0}^{t}A_{-1}^{-1}T_{-1}(t-s)Bu'(s)\,ds\notag\\
&=-A_{-1}^{-1}Bu(t)+\int_{0}^{t}T(t-s)A_{-1}^{-1}Bu'(s)\,ds\notag\\
&=-A_{-1}^{-1}Bu(t)+\big[-A_{-1}^{-2}T_{-1}(t-s)Bu'(s)\big]^t_0+\int_{0}^{t}T(t-s)A_{-1}^{-2}Bu''(s)\,ds\notag\\
&={-A_{-1}^{-1}Bu(t)}-{A_{-1}^{-2}Bu'(t)}+{A^{-1}\int_{0}^{t}T(t-s)A_{-1}^{-1}Bu''(s)\,ds}.
\end{align}
Here the first term in \eqref{daz} belongs to $Z=D(C)$ thanks to the admissibility of the triple $(A,B,C)$, while the second and the third term are in $D(A)\subseteq Z$. Hence, $(\mathcal{F}_{\infty}u)(t)$ is well-defined for all $u\in\mathcal{D}$. 

\smallskip
Next we show that also $(\F|u|)(t)$ is well-defined for all $u\in\mathcal{D}$. By definition, every $u\in\mathcal{D}$ is a $\R$-valued compactly supported function with $\text{supp}(u)\subset[a,b]$ for some $0<a<b<+\infty$. Therefore, it can be represented as
\begin{equation*} 
u:=\sum_{k=1}^{N}u_ke_k,
\end{equation*}
where $e_k=(\delta_{ik})_{i=1}^N\in\mathbb{R}^N$ and $u_k\in C_c^2((0,+\infty),\mathbb{R})$ with $\text{supp}(u_k)\subset[a,b]$ for any $k=1,\dots,N$. Moreover,  we have $|u|=\sum_{k=1}^{N}|u_k|\,e_k$.

Now fix $k\in\{1,\dots,N\}$ and consider the closed set $\mathcal{N}:=\{s\in[0,+\infty):\,u_k(s)=0\}$
and let $\mathcal{O}:=(0,+\infty)\setminus\mathcal{N}$. Since $\mathcal{O}$ is an open set, there exists $J\subseteq\mathbb{N}$ such that 
$$\mathcal{O}=\dot\bigcup_{j\in J}I_j,$$
where $I_j:=(a_j,b_j)$ for $j\in J$ are pairwise disjoint. Notice that
\begin{equation}\label{c2}
|u_k|_{|[a_j,b_j]}\in C^2(I_j,\mathbb{R})\quad\text{and\quad$u_k(a_j)=u_k(b_j)=0$ for all $j\in J$.}
\end{equation}
Next, for $t>0$ we define the set $J_t:=\{j\in J:\,b_j\le t\}$.
Then either $t\in\mathcal{N}$, or there exists (a unique) $j_t\in J\setminus J_t$ such that $t\in(a_{j_t},b_{j_t})$. 
Therefore, 
\begin{equation}\label{u2}
\begin{split}
	\int_{0}^{t}T_{-1}(t-s)B|u_k(s)|e_k\,ds
	&=\int_{\mathcal{O}\cap(0,t)}T_{-1}(t-s)B|u_k(s)|e_k\,ds\\
	&=\sum_{j\in J_t}\int_{a_j}^{b_j}T_{-1}(t-s)B|u_k(s)|e_k\,ds\,+\,R,
\end{split}
\end{equation}
where 
\begin{equation*} 
R:=\begin{cases}
	0,&\text{if}\,\,t\in\mathcal{N},\\
	\int_{a_{j_t}}^{t}T_{-1}(t-s)B|u_k(s)|e_k\,ds,&\text{if}\,\,t\in(a_{j_t},b_{j_t}).
\end{cases}
\end{equation*}
By \eqref{c2} we have $v_k:=|u_k|_{|[a_j,b_j]}\in C^2(I_j,\mathbb{R})$ and $v_k(a_j)=v_k(b_j)=0$ for any $j\in J_t$. Hence, integrating by parts twice as in \eqref{daz} it follows that
\begin{align*} 
	\int_{a_j}^{b_j}T_{-1}(t-s)&Bv_k(s)e_k\,ds=\\
	&A^{-1}\Bigl(T(t-a_j)A_{-1}^{-1}Bv_k'(a_j)e_k-T(t-b_j)A_{-1}^{-1}Bv_k'(b_j)e_k\Bigr)\\
	+&A^{-1}\int_{a_j}^{b_j}T(t-s)A_{-1}^{-1}Bv_k''(s)e_k\,ds.
\end{align*}
Therefore, the series in \eqref{u2}  converges to an element in $D(A)\subseteq Z$. 
As far as $R$ is concerned, if $t\in(a_{j_t},b_{j_t})$, we recall that $v_k:=|u_k|_{|[a_{j_t},t]}\in C^2((a_{j_t},t),\mathbb{R})$ and $v_k(a_{j_t})=0$. Hence, integrating again twice by parts, we obtain
\begin{equation}\label{eq:ipp-R}
\begin{split}
	\int_{a_{j_t}}^{t}T_{-1}(t-s)B&v_k(s)e_k\,ds=-A_{-1}^{-1}Bv_k(t)e_k\\
	+&A^{-1}\Big(T_{-1}(t-a_{j_t})A_{-1}^{-1}Bv_k'(a_{j_t})e_k-A_{-1}^{-1}Bv_k'(t)e_k\Big)\\
	+&A^{-1}\int_{a_{j_t}}^{t}T(t-s)A_{-1}^{-1}Bv_k''(s)e_k\,ds.
\end{split}
\end{equation}
Since the triple $(A,B,C)$ is compatible, the first term on the right-hand-side of \eqref{eq:ipp-R} belongs to $Z$ and the remaining ones to $D(A)$, hence the integral belongs to $Z$. By \eqref{u2} this implies that $\int_{0}^{t}T_{-1}(t-s)B|u_k(s)|e_k\,ds\in Z$ for all $k=1,\dots,N$. Hence, $(\F|u|)(t)$ is well-defined for all $u\in\mathcal{D}$. 
Since $u^+=\frac{|u|+u}{2}$ and $u^-=\frac{|u|-u}{2}$ this also implies that $(\F u^+)(t)$ and $(\F u^-)(t)$ are well-defined for all $u\in\mathcal{D}$. 

\smallskip
Next we verify that $(\F u)(t)\ge0$ for all $u\in|\sD|$. Using the positivity of $(T(t))_{t\ge0}$ and $\lambda R(\lambda,A_{-1})B$ for $\lambda>\max\{0,\spb(A)\}$ we conclude
\begin{align*}
0&\le \int_{0}^{t}T(t-s)\lambda R(\lambda,A_{-1})Bu(s)\,ds
\le\int_{0}^{t}\lambda R(\lambda,A_{-1})T_{-1}(s) B\ub\,ds\\
&\le\int_{0}^{+\infty}\lambda R(\lambda,A_{-1})T_{-1}(s) B\ub\,ds
=-\lambda R(\lambda,A)A_{-1}^{-1} B\ub
\end{align*} 
for all $\lambda>\max\{0,\spb(A)\}$, where $\bar u$ is given by \eqref{hatu}. 
Since $\lambda R(\lambda,A_{-1})\to Id$ in $X$ and $X_{-1}$ we conclude that $-\lambda R(\lambda,A)A_{-1}^{-1} B\ub\to -A_{-1}^{-1}B\ub$ and, by \cite[Prop.A.3]{EN:00},
\[
\lambda R(\lambda,A_{-1})T_{-1}(s)Bu(t-s)\to T_{-1}(s)Bu(t-s)
\quad\text{as $\lambda\to+\infty$}
\]
uniformly for $s\in[0,t]$ in $X_{-1}$. Hence, passing to the limit as $\lambda\to+\infty$ we obtain
\begin{equation*}
0\le\int_{0}^{t}T_{-1}(t-s)Bu(s)\,ds\le-A_{-1}^{-1}B\ub.
\end{equation*}
Multiplying this inequality by the positive operator $C:Z\to\R$ from the left we conclude that 
\begin{equation}\label{eq:Fi-pos}
0\le (\F u)(t)\le-CA_{-1}^{-1}B\ub\quad \text{for all $u\in|\sD|$}
\end{equation}
as claimed. 

\smallbreak
We proceed by proving (i). Let $u\in\mathcal{D}$. First we verify that $\mathcal{F}_{\infty}u\in\cwor$.

Clearly, $(\mathcal{F}_{\infty}u)(0)=0$. Moreover,  by \eqref{daz} we have
\begin{equation}\label{ft}
\begin{split}
	(\mathcal{F}_{\infty}u)(t)=-CA_{-1}^{-1}Bu(t)&-CA_{-1}^{-2}Bu'(t)\\&+CA^{-1}\int_{0}^{t}T(t-s)A_{-1}^{-1}Bu''(s)ds,
\end{split}
\end{equation}
which defines a continuous map $[0,+\infty)\ni t\mapsto(\mathcal{F}_{\infty}u)(t)\in\R$. Next we show that $\lim_{t\to+\infty}(\mathcal{F}_\infty u)(t)=0$. Since $u$ is compactly supported there exists $b>0$ such that $u(t)=0$ for all $t\ge b$. Clearly, this implies $u'(t)=u''(t)=0$ for all $t>b$, and therefore
\begin{equation*}
\begin{split}
	\lim_{t\to+\infty}CA_{-1}^{-1}Bu(t)=0\quad\text{and}\quad
	\lim_{t\to+\infty} CA_{-1}^{-2}Bu'(t)=0.
\end{split}
\end{equation*}
Moreover, for $t>b$ the last term in \eqref{ft} becomes
\begin{equation*}
\begin{split}
	CA^{-1}\int_{0}^{t}T(t-s)A_{-1}^{-1}Bu''(s)\,ds&=CA^{-1}\int_{0}^{b}T(t-s)A_{-1}^{-1}Bu''(s)\,ds\\&=CA^{-1}T(t-b)\int_{0}^{b}T(b-s)A_{-1}^{-1}Bu''(s)\,ds.
\end{split}
\end{equation*}
Here the integral term is independent of $t$, while $\lVert T(t-b)\rVert\to0$ as $t\to+\infty$. This implies that
$$\lim_{t\to+\infty}CA^{-1}\int_{0}^{t}T(t-s)A_{-1}^{-1}Bu''(s)\,ds=0$$
and hence $\lim_{t\to+\infty}(\mathcal{F}_\infty u)(t)=0$ for any $u\in\mathcal{D}$. 

\smallbreak
Next, using \eqref{eq:Fi-pos} we obtain for any $u=u^+-u^-\in\mathcal{D}$ and $t\ge0$
\begin{align}\label{baru}
\bigl|(\mathcal{F}_\infty u)\bigr|(t)
&=\left\lvert(\F u^+)(t)-(\F u^-)(t)\right\rvert\notag\\
	 &\le (\F u^+)(t) + (\F u^-)(t)\\
	&=\bigl(\F|u|\bigr)(t)
\le-C A_{-1}^{-1}B\,\bar{u} .\notag
\end{align}
This yields that for all $u\in\mathcal{D}$ and $t\ge0$ we have
\begin{equation*}
\bigl\lVert(\mathcal{F}_\infty u)(t)\bigr\rVert_{\W}\le\bigl\lVert C A_{-1}^{-1}B\bar{u}\bigr\rVert_{\W}\le\lVert CA_{-1}^{-1}B\rVert_{\sL(\W)}\cdot\lVert u\rVert_\infty.
\end{equation*}
Passing to the supremum over all $t\ge0$ we obtain
\begin{equation}\label{estf2}
\lVert\mathcal{F}_\infty u\rVert_\infty\le\lVert C A_{-1}^{-1}B\rVert_{\sL(\W)}\cdot\lVert u\rVert_\infty\quad\text{for all}\,\,u\in\mathcal{D}.
\end{equation}
Summing up, this shows that $\mathcal{F}_\infty:\mathcal{D}\subset\cwor\rightarrow\cwor$ is bounded, positive and satisfies \eqref{estf2}. Hence, there exists a unique positive  extension $\mathcal{F}_\infty\in\sL(\cwor)$ satisfying
\eqref{eq:Fi-C0} for $n=1$. 

\smallbreak
In order to verify \eqref{eq:Fi-C0} for arbitrary $n\in\NN$ we proceed by induction and first show that 
\begin{equation}\label{cabn-C0}
\big(\F^n |u|\big)(t)\le\bigl(-C A_{-1}^{-1}B\bigr)^n\bar{u}
\quad\text{for all $u\in\sD$, $t\ge0$ and $n\in\NN$},
\end{equation}
where $\bar u$ is given by \eqref{hatu}.
For $n=1$ this was shown in \eqref{baru}.  Now assume it holds for some $n\in\NN$. Then
\begin{equation*}
\begin{split}
\big(\F^{n+1} |u|\big)(t)
&=\F\big(\F^n |u|\big)(t)=C\int_{0}^{t}T_{-1}(s)B\bigl(\F^n |u|\bigr)(t-s)\,ds\\
&\le C\int_{0}^{+\infty}T_{-1}(s)B\bigl(-C A_{-1}^{-1}B\bigr)^n\bar{u}\,ds\\
&=-C A_{-1}^{-1}B\bigl(-C A_{-1}^{-1}B\bigr)^n\bar{u}=\bigl(-C A_{-1}^{-1}B\bigr)^{n+1}\bar{u}
\end{split}
\end{equation*}
proving \eqref{cabn-C0}.
Arguing as in \eqref{eq:est-Fin} this gives $\|(\F^n u)(t)\|_{\W}\le\|(C A_{-1}^{-1}B)^n\|_{\sL(\W)}\cdot\|u\|_\infty$ for every $t\ge0$. Passing  the supremum over all $t\ge0$ this proves \eqref{eq:Fi-C0} and concludes the proof of (i).

\smallbreak
(ii). Take $u\in\mathcal{D}$.  Then by Fubini's Theorem we obtain
\begin{align*} 
\int_0^{+\infty}\int_0^t T_{-1}(t-s)Bu(s)\,ds\,dt
&=\int_0^{+\infty}\int_s^{+\infty} T_{-1}(t-s)Bu(s)\,dt\,ds\\
&=\int_0^{+\infty}\int_0^{+\infty} T_{-1}(r)Bu(s)\,dr\,ds\\
&=-A_{-1}^{-1}B\int_0^{+\infty} u(s)\,ds.
\end{align*}
Multiplying this equation by $C$ from the left, using integration by parts and \Cref{binf} we obtain for $u\in\mathcal{D}$ that
\begin{align}\label{eq:Fi1-L1}
-CA_{-1}^{-1}B\int_0^{+\infty}\!\!\! u(s)\,ds\notag
&=C\int_0^{+\infty}\int_0^t T_{-1}(t-s)Bu(s)\,ds\,dt\\\notag
&=C\int_0^{+\infty}\!\!\!-A_{-1}^{-1}Bu(t)+\int_0^t A_{-1}^{-1}T_{-1}(t-s)Bu'(s)\,ds\,dt\\\notag
&=-CA_{-1}^{-1}B\int_0^{+\infty}\!\!\!u(t)\,dt+CA^{-1}\int_0^{+\infty}\!\!\int_0^t T_{-1}(t-s)Bu'(s)\,ds\,dt\\\notag
&=\int_0^{+\infty}\!\!\!-CA_{-1}^{-1}Bu(t)\,dt+\int_0^{+\infty}CA^{-1}\int_0^t T_{-1}(t-s)Bu'(s)\,ds\,dt\\\notag
&=\int_0^{+\infty}\!\!\!-CA_{-1}^{-1}\Bigl(Bu(t)-\int_0^t T_{-1}(t-s)Bu'(s)\,ds\Bigr)\,dt\\\notag
&=\int_0^{+\infty}C\int_0^t T_{-1}(t-s)Bu(s)\,ds\,dt\\
&=\int_0^{+\infty}(\F u)(t)\,dt.
\end{align}
Next, since $\V=(\R,|\cdot|_1)$ is an AL-space, we conclude for $u\in\mathcal{D}$
\begin{equation}\label{bb}
\begin{split}
\lVert\mathcal{F}_\infty u\rVert_{1}
&\le \int_0^{+\infty}\|(\F|u|)(t)\|_{\V}dt
=\Bigl\|\int_0^{+\infty}(\F|u|)(t)dt\Bigr\|_{\V}\\
&=\Bigl\lVert -C A_{-1}^{-1}B \int_0^{+\infty}|u(t)|dt\Bigr\rVert_{\V}
\le\lVert C A_{-1}^{-1}B\rVert_{\sL(\V)}\cdot\lVert u\rVert_{1}.
\end{split}
\end{equation}
Summing up, we showed that $(\mathcal{F}_\infty u)(t)$ given by \eqref{fto} defines a positive linear operator $\mathcal{F}_\infty:\mathcal{D}\subset\lrv\to\lrv$ satisfying \eqref{bb}. Hence, $\mathcal{F}_\infty$ possesses a unique positive extension $\mathcal{F}_\infty\in\sL(\lrv)$ satisfying \eqref{eq:Fi-L1}  for $n=1$.

In order to verify \eqref{eq:Fi-L1} for arbitrary $n\in\NN$ we proceed by induction and first show
that for any $u\in\lrv$ and $t\ge0$ we have
\begin{equation}\label{cabn-L1}
\int_{0}^{+\infty}\bigl\lvert(\F^nu)(t)\bigr\rvert\,dt\le\bigl(-CA_{-1}^{-1}B\bigr)^n\int_{0}^{+\infty}\bigl|u(t)\bigr|\,dt\quad\text{for all}\,\,n\in\mathbb{N}.
\end{equation}
Since $\mathcal{D}\subset\lrv$ is dense, we first observe that \eqref{eq:Fi1-L1} holds for all $u\in\lrv$. Using the positivity of $\F$ this implies
\[
\int_0^{+\infty}\bigl|(\F u)\bigr|(t)\,dt
\le \int_0^{+\infty}\bigl(\F |u|\bigr)(t)\,dt
\le-CA_{-1}^{-1}B\int_0^{+\infty}\bigl|u(t)\bigr|\,dt,
\]
i.e., \eqref{cabn-L1} for $n=1$. The general case $n\in\NN$ then follows by induction as in \eqref{eq:Fin-L1}. Proceeding as in \eqref{eq:Fin-L1-n} this implies \eqref{eq:Fi-L1}.
This completes the proof of \Cref{lemmaoo}.
\end{proof}

\begin{lemma}\label{lt-p1}
Let $(A,B,C)$ be compatible and $B:\R\to\Xmo$ and $C:Z\to \R$ be positive. Then for all $1\le p<+\infty$ the operator $\F$ defined by \eqref{fto} possesses a unique bounded extension (still denoted by $\F$)  $\F\in\sL(\lprr)$ satisfying the estimate
\begin{equation}\label{foon3}
\bigl\lVert\F^{n}\bigr\rVert_{\sL(\rL^p)}
\le\bigl\lVert (CA_{-1}^{-1}B)^n\bigr\rVert_{\sL(\V)}^{1/p}\cdot\bigl\lVert(CA_{-1}^{-1}B)^n\bigr\rVert_{\sL(\W)}^{1-1/p}
\quad\text{for every }n\in\NN.
\end{equation}
\end{lemma}

\begin{proof} By \Cref{lemmaoo}, $\F$ defined by \eqref{fto} on $\sD=C_c^2((0,+\infty),\R)$ possesses unique bounded extensions to $\ceor$ and to $\lrr$. Now let $u\in\mathcal{C}:=C_c((0,+\infty),\R)$. Then there exists a sequence $(u_n)_{n\in\NN}\subset\mathcal{D}$ converging uniformly to $u$. Clearly, this implies that $u_n\to u$ in $\ceor$ and in $\lrr$ as $n\to+\infty$, hence both extensions coincide on $\mathcal{C}$. Moreover, again by \Cref{lemmaoo} these extensions satisfy the estimates
\begin{equation*} 
\lVert\mathcal{F}_\infty\rVert_{\sL(C_0)}\le\bigl\lVert CA_{-1}^{-1}B\bigr\rVert_{\sL(\W)}
\quad\text{and}\quad
\lVert\mathcal{F}_\infty\rVert_{\sL(\rL^1)}\le\bigl\lVert CA_{-1}^{-1}B\bigr\rVert_{\sL(\V)}.
\end{equation*}
By \Cref{RT} this implies that 
\[\F:\mathcal{C}\subset\lprr\to\lrr\cap\ceor\subset\lprr\]
possesses a unique bounded extension $\F\in\sL(\lprr)$.  The estimate \eqref{foon3} then follows from \Cref{lemmaoo}.(i) and (ii)
by applying again \Cref{RT} to $\F^n$.
\end{proof}

After having defined the operator $\F\in\sL(\lprr)$, we now investigate the invertibility of $Id-\F$ and its Laplace transform. 

\begin{lemma}\label{lt-p2}
Let $(A,B,C)$ be compatible, $B:\R\to\Xmo$ and $C:Z\to\R$ be positive, and assume that $\sr(CA_{-1}^{-1}B)<1$. Then for any $\lambda>0$ one has $1\in\rho(\F)\cap\rho(CR(\lambda,A_{-1})B)$ and
\begin{equation}\label{lap}
\mathscr{L}\big((Id-\mathcal{F}_\infty)^{-1}u\big)(\lambda)=\big(Id-CR(\lambda,A_{-1})B\big)^{-1}\cdot\hat{u}(\lambda)
\end{equation}
for all $u\in\lprr$, $1\le p<\infty$.
\end{lemma} 

\begin{proof}
By the same reasoning used to show \eqref{spes} it follows that $\sr(CR(\lambda,A_{-1})B)<1$ for all $\lambda>0$, 
hence $Id-C R(\lambda,A_{-1})B:\R\to\R$ is invertible for any $\lambda>0$. Since by assumption
\[
\sr(CA_{-1}^{-1}B)=\lim_{n\to+\infty}\|(CA_{-1}^{-1}B)^n\|_{\sL(\V)}^{1/n}=\lim_{n\to+\infty}\|(CA_{-1}^{-1}B)^n\|_{\sL(\W)}^{1/n}<1,
\]
\eqref{foon3} implies that $\F\in\sL(\lprr)$ satisfies 
\[
\sr(\F)=\lim_{n\to+\infty}\bigl\lVert\F^{n}\bigr\rVert_{\sL(\rL^p)}^{1/n}<1.
\]
Hence, the operator $Id-\mathcal{F}_\infty:\lprr\to\lprr$ is invertible.
The identity \eqref{lap} then follows from \cite[Lem.13]{ABE:14}.
\end{proof}

We have now all the necessary tools to prove the main result of this section.
\begin{proof}[Proof of \Cref{PT}]
We define an operator family $(S(t))_{t\ge0}\subset\sL(X)$ and prove that it is a  $C_0$-semigroups with generator $A_{BC}$ given by \eqref{abc}.

By \Cref{lt-p2}, $Id-\mathcal{F}_\infty$ is invertible, and therefore we can define
\begin{equation}\label{semigg}
S(t):=T(t)+\mathcal{B}_t(Id-\mathcal{F}_\infty)^{-1}\mathcal{C}_\infty\in\sL(X)
\quad\text{for any $t\ge0$.}
\end{equation}
From \Cref{toop} it follows that this operator family is strongly continuous on $X$ and uniformly bounded. 
The same computations as in \cite[Proof of Thm.10]{ABE:14} then show that
\begin{equation*}
\mathscr{L}\bigl(S(\cdot)x\bigr)(\lambda)=R(\lambda,A_{BC})x\quad\text{for all $x\in X$ and $\lambda>0$.}
\end{equation*}
By \cite[Thm.3.1.7]{ABHN:11}, this implies that the operator family \eqref{semigg} defines a positive  $C_0$-semigroup on $X$ with generator $A_{BC}$. Finally, the variation of parameters \eqref{VP} follows as in \cite[Thm.10]{ABE:14}.
\end{proof}

\subsection{Structured Perturbation via Domination}\label{PvD}

In this subsection we relax the positivity assumptions on the operators $B:U\to\Xmo$ and $C:Z\to U$ and only suppose that they are dominated by positive operators $\widetilde{B}:U\to\Xmo$ and $\widetilde{C}:Z\to U$ such that $A_{\widetilde{B}\widetilde{C}}$ is a generator. More precisely, we impose the following conditions. Note that in this subsection we do not need $X$ or $U$ to be an AM- or AL-space. 

\begin{assumption}\label[assumption]{MA-dom} In the situation of \Cref{MA}, assume that there exist positive operators $B_+,\,B_-:U\to\Xmo$ and $\widetilde{C}:Z\to U$ such that
\begin{enumerate}[label=(\roman*)]
	\item $|Cx|\le\widetilde{C}x$ for all $x\in Z_+=Z\cap X_+$,
	\item $B=B_+-B_-$, i.e., $B$ is regular.
\end{enumerate}
\end{assumption} 

Then the following holds where we put $\widetilde{B}:=B_++B_-$.

\begin{theorem}\label{thm:Dom}
Let \Cref{MA,MA-dom} be satisfied and suppose that the triples $(A,B,C)$ and $(A,\widetilde{B},\widetilde{C})$ are compatible. If 
$A_{\widetilde{B}\widetilde{C}}$ generates a positive $C_0$-semigroup  $(\widetilde{S}(t))_{t\ge0}$ on $X$
and 
$\sr(\widetilde{C}R(\lambda,A_{-1})\widetilde{B})<1$ for some $\lambda>\omega_0(A)$, then the operator
\begin{equation*} 
	\begin{aligned}
\hspace{3cm} A_{BC}&:=\bigl(A_{-1}+BC\bigr)|_{X},\\
D(A_{BC})&:=\bigl\{x\in Z: A_{-1}x+BCx\in X\bigr\}
 \end{aligned}
 \end{equation*}
generates a  $C_0$-semigroup $(S(t))_{t\ge0}$ satisfying $|S(t)x|\le\widetilde{S}(t)x$ for all $t\ge0$, $x\ge0$. 
\end{theorem}

We refer to \Cref{subsec:HE,subsec:uP-FD,subsec:r1p-FD} for concrete applications of this theorem.

\smallbreak
For its proof we need some preparation. As usual, we might assume in the sequel that $\omega_0(A)<0$ and $\sr(\widetilde{C}A_{-1}^{-1}\widetilde{B})<1$. 

\begin{lemma}\label{wlemma}
	If the triples $(A,B,C)$ and $(A,\widetilde{B},\widetilde{C})$ are compatible then for all $\lambda>0$ it holds
	\begin{enumerate}[label=(\roman*)]
		\item $CR(\lambda,A)\in\mathcal{L}(X,U)$ and $\lVert CR(\lambda,A)\rVert_{\mathcal{L}(X,U)}\le\lVert \widetilde{C}R(\lambda,A)\rVert_{\mathcal{L}(X,U)}$,
		\item $R(\lambda,A_{-1})B\in\mathcal{L}(U,X)$ and $\lVert R(\lambda,A_{-1})B\rVert_{\mathcal{L}(U,X)}\le\lVert R(\lambda,A_{-1})\widetilde{B}\rVert_{\mathcal{L}(U,X)}$,
		\item $CR(\lambda,A_{-1})B\in\mathcal{L}(U)$ and $\lVert (CR(\lambda,A_{-1})B)\rVert_{\mathcal{L}(U)}\le\lVert (CR(\lambda,A_{-1})B)\rVert_{\mathcal{L}(U)}$,
		\item $\sr(CR(\lambda,A_{-1})B)\le \sr(\widetilde{C}R(\lambda,A)\widetilde{B})\le \sr(\widetilde{C}A_{-1}^{-1}\widetilde{B})$.
	\end{enumerate}
\end{lemma}

\begin{proof}
	(i). Since $R(\lambda,A)$ is positive for $\lambda>0$, using \Cref{MA-dom}.(i) we obtain for all $x\in X_+$ and $\lambda>0$
	\begin{equation*}
		\bigl|CR(\lambda,A)x\bigr|\le\widetilde{C}R(\lambda,A)x.
	\end{equation*}
Hence, for all $x\in X$ and $\lambda>0$ we have
\begin{equation}\label{eq:CRlA}
\begin{aligned}
0\le\bigl|CR(\lambda,A)x\bigr|
&\le\bigl|CR(\lambda,A)x^+\bigr|+\bigl|CR(\lambda,A)x^-\bigr|\\
&\le\widetilde{C}R(\lambda,A)x^++\widetilde{C}R(\lambda,A)x^-
=\widetilde{C}R(\lambda,A)|x|.
\end{aligned}
\end{equation}
Applying the norm and taking the supremum in the unit ball of $U$ gives (i).

\smallbreak
(ii). By \Cref{MA-dom}.(ii), $B_\pm$ are positive linear operators and satisfy $B=B_+-B_-$ and $\widetilde{B}=B_++B_-$. Moreover,  the triples $(A,B_{\pm},C)$ and $(A,B_{\pm},\widetilde{C})$ are compatible. Hence, for all $u\in U$ and $\lambda>0$ using \Cref{L:B-pos} we conclude
\begin{equation*} 
\begin{aligned}
0\le\bigl|R(\lambda,A_{-1})Bu\bigr|
&=\bigl|R(\lambda,A_{-1})\left(B_+-B_-\right)u\bigr|\\
&\le\bigl|R(\lambda,A_{-1})B_+u\bigr|+\bigl|R(\lambda,A_{-1})B_-u\bigr|\\
&\le\bigl|R(\lambda,A_{-1})B_+u^+\bigr|+\bigl|R(\lambda,A_{-1})B_+u^-\bigr|\\
&\quad+\bigl|R(\lambda,A_{-1})B_-u^+\bigr|+\bigl|R(\lambda,A_{-1})B_-u^-\bigr|\\
&\le R(\lambda,A_{-1})\bigl(B_+|u|+B_-|u|\bigr)=R(\lambda,A_{-1})\widetilde{B}|u|.
\end{aligned}
\end{equation*}
Applying the norm and taking the supremum in the unit ball of $U$ gives (ii).
	
\smallbreak
(iii)--(iv). Using \Cref{MA-dom}, \Cref{L:B-pos} and the resolvent identity we obtain for all $u\in U$ and $\lambda>0$
\begin{equation}\label{crbu0}
\begin{aligned}
0\le \bigl|CR(\lambda,A_{-1})Bu|
&=\bigl|CR(\lambda,A_{-1})\bigl(B_+u-B_-u\bigr)\bigr|\\
&\le\bigl|CR(\lambda,A_{-1})B_+u\bigr|+\bigl|CR(\lambda,A_{-1})B_-u\bigr|\\
&\le\bigl|CR(\lambda,A_{-1})B_+u^+\bigr|+\bigl|CR(\lambda,A_{-1})B_+u^-\bigr|\\
&\quad+\bigl|CR(\lambda,A_{-1})B_-u^+\bigr|+\bigl|CR(\lambda,A_{-1})B_-u^-\bigr|\\
&\le\widetilde{C}R(\lambda,A_{-1})\bigl(B_+u^++B_+u^-+B_-u^++B_-u^-\bigr)\\
&=\widetilde{C}R(\lambda,A_{-1})\bigl(B_+|u|+B_-|u|\bigr)=\widetilde{C}R(\lambda,A_{-1})\widetilde{B}|u|\\
&=\widetilde{C}R(0,A_{-1})\widetilde{B}|u|-\widetilde{C}\lambda R(\lambda,A)R(0,A_{-1})\widetilde{B}\\
&\le  -\widetilde{C}A_{-1}^{-1}\widetilde{B}|u|.
\end{aligned}
\end{equation}
Assertions (iii) and (iv) then follow from \eqref{crbu0} by  \Cref{spr}.
\end{proof}

\begin{lemma}\label{L:ABC-cdd} 
Let $A_{\widetilde{B}\widetilde{C}}$ generate a positive semigroup on $X$. If $\sr(\widetilde{C}R(\lambda,A_{-1}\widetilde{B})<1$ for some $\lambda>0$ then  
\begin{enumerate}[label=(\roman*)]
\item $A_{BC}$ is closed and densely defined,
\item $\lambda\in\rho(A_{BC})\cap\rho(A_{\widetilde{B}\widetilde{C}})$,
\item $\|R(\lambda,A_{BC})^n\|_{\sL(X)}\le\|R(\lambda,A_{\widetilde{B}\widetilde{C}})^n\|_{\sL(X)}$ for all $n\in\NN$.
\end{enumerate}
\end{lemma}

\begin{proof}
By \Cref{wlemma}.(iv) the hypothesis $\sr(\widetilde{C}R(\lambda,A_{-1}\widetilde{B})<1$ implies that $1\in\rho(CR(\lambda,A_{-1}B))\cap\rho(\widetilde{C}R(\lambda,A_{-1}\widetilde{B}))$. Hence, by \cite[Thm.2.3.(d)]{AE:18} we conclude $\lambda\in\rho(A_{BC})\cap\rho(A_{\widetilde{B}\widetilde{C}})$ and
\begin{align}
R(\lambda,A_{BC})&=R(\lambda,A)+R(\lambda,A_{-1})B\bigl(Id-CR(\lambda,A_{-1})B\bigr)^{-1}CR(\lambda,A),\label{eq:RlABC1}\\
R(\lambda,A_{\widetilde{B}\widetilde{C}})&=R(\lambda,A)+R(\lambda,A_{-1})\widetilde{B}\bigl(Id-\widetilde{C}R(\lambda,A_{-1})\widetilde{B}\bigr)^{-1}\widetilde{C}R(\lambda,A)\label{eq:RlABC2},
\end{align}
which proves (ii). Moreover, this implies that $A_{BC}$ is closed. To show that $D(A_{BC})$ is dense in $X$ we fix $x\in X$. Then $\lambda R(\lambda,A)x\to x$ and $\lambda R(\lambda,A_{\widetilde{B}\widetilde{C}})x\to x$ in $X$ as $\lambda\to+\infty$. From \eqref{eq:RlABC2} it follows
\[\lambda R(\lambda,A_{-1})\widetilde{B}\bigl(Id-\widetilde{C}R(\lambda,A_{-1})\widetilde{B}\bigr)^{-1}\widetilde{C}R(\lambda,A)x\to 0\quad\text{as}\,\,\lambda\to+\infty.\]
Using \eqref{eq:CRlA}--\eqref{crbu0} and \Cref{spr}.(i) this yields for arbitrary $x\in X$
\begin{align}\label{eq:RABC}
0\le\bigl|\lambda R(\lambda,A_{-1})B\bigl(Id&-CR(\lambda,A_{-1})B\bigr)^{-1}CR(\lambda,A)x\bigr|\notag\\
&=\biggl|\lambda R(\lambda,A_{-1})B\cdot\sum_{n=0}^\infty \bigl(CR(\lambda,A_{-1})B\bigr)^n\cdot CR(\lambda,A)x\biggr|\notag\\
&\le\lambda R(\lambda,A_{-1})\widetilde{B}\cdot\sum_{n=0}^\infty \bigl(\widetilde{C}R(\lambda,A_{-1})\widetilde{B}\bigr)^n\cdot \widetilde{C}R(\lambda,A)|x|\notag\\
&=\lambda R(\lambda,A_{-1})\widetilde{B}\bigl(Id-\widetilde{C}R(\lambda,A_{-1})\widetilde{B}\bigr)^{-1}\widetilde{C}R(\lambda,A)|x|\\
&\to 0\quad\text{as $\lambda\to+\infty$}\notag.
\end{align}
Hence, from \eqref{eq:RlABC1} it follows $D(A_{BC})\ni\lambda R(\lambda,A_{BC})x\to x$ as $\lambda\to+\infty$. Since $x\in X$ was arbitrary, this shows $A_{BC}$ is densely defined as claimed. It only remains to show (iii). Using \eqref{eq:RlABC1} and \eqref{eq:RlABC2}, the same reasoning as in \eqref{eq:RABC} implies
\begin{equation}\label{eq:dom-Rla}
\bigl|R(\lambda,A_{BC})x\bigr|\le R(\lambda,A_{\widetilde{B}\widetilde{C}})x
\end{equation}
for all $x\in X_+$. By \Cref{spr}.(ii) this proves (iii).
\end{proof}

\begin{proof}[Proof of \Cref{thm:Dom}]
Without loss of generality we may assume $\omega_0(A)<0$ and $\sr(\widetilde{C}A_{-1}^{-1}\widetilde{B})<1$. 
Then by \Cref{L:ABC-cdd}.(i) the operator $A_{BC}$ is closed and densely defined. Moreover, by \Cref{wlemma}.(iv) and \Cref{L:ABC-cdd}.(ii) we know that $\RR_+\subset\rho(A_{BC})\cap\rho(A_{\widetilde{B}\widetilde{C}})$. Since by assumption $A_{\widetilde{B}\widetilde{C}}$ is a generator it verifies the Hille--Yosida estimates for the $n$-th powers of its resolvent. By \Cref{L:ABC-cdd}.(iii) the same is true for $A_{BC}$ which implies that it is a generator as well. Finally, the domination $|S(t)x|\le\widetilde{S}(t)x$ for all $t\ge0$, $x\ge0$ follows from \eqref{eq:dom-Rla} by the Post--Widder inversion formula, cf.\ \cite[Cor.III.5.5]{EN:00}.
\end{proof}

\section{Applications}\label{sec:App}
\subsection{Abstract Domain Perturbation of Generators}
As we saw in \Cref{sec:intro}, operators of the form ``$G=A+BC\,$'' naturally arise by closing a linear system. Next we show that also boundary perturbations in the sense of Greiner, cf. \cite{Gre:87}, can be written in this form. 
More precisely, we consider the following abstract framework.

\smallbreak
We start with a Banach space $X$ and a  linear ``maximal operator''\footnote{``Maximal'' in the sense of a ``big'' domain, e.g., a differential operator without boundary conditions.} $A_m:D(A_m)\subseteq X\to X$.
In order to single out a restriction $A$ of $A_m$ we take a Banach spaces $\partial X$, called ``boundary space'', and a linear ``boundary (or trace) operator''
$L:D(A_m)\to\partial X$
and define the operator $A$ with ``abstract Dirichlet boundary conditions'' by
\begin{equation}\label{eq:def-A-kerK}
A\subset A_m,\quad D(A):=\ker(L)=\bigl\{x\in D(A_m): Lx=0\bigr\}.
\end{equation}
Next we perturb the domain of $A$ in the following way.
For a Banach space $Z$ satisfying $D(A_m)\subseteq Z$ and $X_1\hookrightarrow Z\hookrightarrow X$ with continuous injections, and an operator $\Phi\in\sL(Z,\partial X)$ we define%
\begin{equation}\label{eq:G-one-dim}
A^\Phi\subset A_m,\quad D(A^\Phi):=\ker(L-\Phi)=\bigl\{x\in D(A_m): Lx=\Phi x\bigr\}.
\end{equation}
Hence, $A^\Phi$ can be considered as domain perturbation of the operator $A$.
We note that in \cite{Gre:87} the operator $\Phi:X\to\partial X$ has to be bounded.
As we will see this example fits into our framework also for unbounded operators $\Phi$. To proceed we need the following result. For a proof we refer to \cite[Sect.3]{AE:18}.

\begin{lemma}
\label{lem:Dir-op-ex}
Assume that $A:D(A)\subseteq X\to X$ given in \eqref{eq:def-A-kerK} has non-empty resolvent set $\rho(A)$. Moreover, suppose that
\begin{enumerate}[label=(\roman*)]
\item $A_m$ or $\binom{A_m}{L}$ is closed and  $L:Z\to\partial X$  is surjective, or
\item for some $\mu\in\mathbb{C}$ the restriction
\[
L|_{\ker(\mu-A_m)}:\ker(\mu-A_m)\to\partial X
\]
is invertible with bounded inverse
\begin{equation}\label{eq:Dir-op}
L_\mu:=\bigl(L|_{\ker(\mu-A_m)}\bigr)^{-1}\in
\sL(\partial X,X).
\end{equation}
\end{enumerate}
Then for all $\lambda\in\rho(A)$
\[
L|_{\ker(\lambda-A_m)}:\ker(\lambda-A_m)\to\partial X
\]
is invertible with bounded inverse given by
\begin{equation}\label{eq:def-L_mu}
L_\lambda=(\mu-A)R(\lambda,A) L_\mu\in
\sL(\partial X,X).
\end{equation}

\end{lemma}
As we see next the so-called abstract \emph{Dirichlet operator}\footnote{This notion is used since $f=L_\mu x_0$ solves the abstract Dirichlet problem $(\mu-A_m)f=0$, $Lf=x_0$.}
$L_\mu$ defined in \eqref{eq:Dir-op} plays a crucial role in the context of this generic example. Note that by \eqref{eq:def-L_mu} $L_\lambda=R(\lambda,A_{-1})(\mu-A_{-1}) L_\mu$, hence
the operator
\begin{equation}\label{eq:def-LA}
L_A:=(\mu-A_{-1})L_\mu=(\lambda-A_{-1})L_\lambda\in\sL\bigl(\partial X,X_{-1}\bigr)
\end{equation}
is independent of $\mu$ and $\lambda\in\rho(A)$. Using this operator we obtain the following representation of $A^\Phi$ from \eqref{eq:G-one-dim}. For a proof we refer again to \cite[Sect.3]{AE:18}.

\begin{lemma}\label{lem:G-as-SWp}
If $L_\lambda:X\to\partial X$ exists for some $\mu\in\rho(A)$ then
\begin{equation}\label{eq:rep-Aphi}
A^\Phi=\bigl(A_{-1}+L_A\cdot\Phi\bigr)\big|_{X}.
\end{equation}
\end{lemma}

Hence, $A^\Phi$ given in  \eqref{eq:G-one-dim}  can be represented as $A_{BC}$ like in \eqref{eq:def-A_BC-i} for $U:=\partial X$
and the operators
\begin{equation}\label{eq:def-B-C}
B:=L_A\in\sL(U,X_{-1})
\quad\text{and}\quad
C:={\Phi}\in\sL(Z,U).
\end{equation}
Moreover, the following holds.

\begin{lemma}\label{cor:rep-G}
The triple $(A,B,C)$ given by \eqref{eq:def-A-kerK}, \eqref{eq:def-B-C} is compatible in the sense of \Cref{def2}, i.e., satisfies
\[
\rg\bigl(R(\lambda,A_{-1}) B\bigr)\subseteq Z\quad\text{for all }\lambda\in\rho(A).
\]
\end{lemma}

\begin{proof}
The claim follows from the inclusions
\begin{equation*}
\rg\bigl(R(\lambda,A_{-1}) B\bigr)=\rg(L_\lambda)=\ker(\lambda-A_m)\subseteq D(A_m)\subseteq Z.
\qedhere
\end{equation*}
\end{proof}

Finally, we apply \Cref{AM} to this situation where we assume $X$ and $\partial X$ to be Banach lattices. Since by \eqref{eq:def-LA} and \eqref{eq:def-B-C} we have $C R(\lambda,A_{-1})B=\Phi L_\lambda$  we obtain the following result.

\begin{corollary}\label{cor:dX-AM}
Let $A:D(A)\subseteq X\to X$ given in \eqref{eq:def-A-kerK} be the generator of a positive  $C_0$-semigroup. Assume that $L_\mu:\partial X\to X$  for $\mu>\omega_0(A)$ and $\Phi:X\to\partial X$ are positive linear operators, where $\partial X$ is an AM-space. If the spectral radius $\sr(\Phi L_\lambda)<1$ for some $\lambda>\omega_0(A)$, then the operator $A^\Phi$ defined in \eqref{eq:G-one-dim}
generates a positive  $C_0$-semigroup on $X$.
\end{corollary}

We remark that in case $\partial X$ is finite dimensional also \Cref{PT} can be applied to this situation allowing for unbounded $\Phi:Z\subset X\to\partial X$.

\subsection{Heat Equation with Boundary Feedback on $\mathbf{L^2(\Omega)}$}\label{subsec:HE}

We give an application of \Cref{cor:dX-AM} and \Cref{thm:Dom}.
To this end we consider the partial differential equation
\begin{equation}\label{eq:HE-bfb}
\begin{cases}
\begin{aligned}
\tfrac\partial{\partial t}u(t,x)&=\Delta u(t,x),\quad &&x\in\Omega,\ t\ge0,\\
u(0,x)&=u_0(x),\quad &&x\in\Omega,\\
u(t,z)&=\int_\Omega \phi(z,x)\,u(t,x)\,dx,\quad &&z\in\partial\Omega,\ t\ge0,
\end{aligned}
\end{cases}
\end{equation}
where
\begin{itemize}
\item $\Omega$ is a bounded domain in $\RR^n$ with $C^2$-boundary $\partial\Omega$,
\item $\Delta$ denotes the Laplace operator on $\Omega$, 
\item $u_0\in \rL^2(\Omega)$, and
$\phi\in \rC(\partial\Omega,\overline{\Omega})$.
\end{itemize}
The well-posedness of a related heat equation with dynamic boundary conditions was studied in \cite[Sect.6]{CENN:03} using a matrix approach on the state space $\mathcal{X}=\rL^2(\Omega)\times \rL^2(\partial\Omega)$. In contrast, here we consider \eqref{eq:HE-bfb} on the space $X=\rL^2(\Omega)$ and show that the associated operator $A^\Phi$ generates a $C_0$-semigroup. To this end we introduce the following spaces and operators:
\begin{itemize}
\item $X:=\rL^2(\Omega)$, $\partial X:=\rL^2(\partial\Omega)$,
\item $A_m:=\Delta:D(A_m)\subset X\to X$, $D(A_m):=\{f\in \rH^{\frac12}(\Omega)\cap \rH^2_{\mathrm{loc}}(\Omega):\Delta f\in X\}$, 
\item $L:D(A_m)\to\partial X$, $Lf:=f|_{\partial\Omega}$ in the sense of traces, cf.\ \cite[Chap.2]{LM:72},
\item $\Phi:X\to\partial X$, $(\Phi f)(z):=\int_\Omega \phi(z,x)\,f(x)\,dx$, $z\in\partial\Omega$, and
\item $A,A^\Phi\subset A_m$, $D(A):=\ker L$, $D(A^\Phi):=\ker(L-\Phi)$.
\end{itemize}
As shown in \cite[Sects.3\&6]{CENN:03}, in this context the Dirichlet operators $L_\lambda:X\to\partial X$ introduced in \Cref{lem:Dir-op-ex} exist and are positive for all $\lambda\ge0$. Moreover, by \cite[Cap.5 \& App.A9]{Tay:96} the operator $A$ generates a positive analytic semigroup on $X$ satisfying $\omega_0(A)<0$.
Hence, by the results of the previous subsection we can represent the operator $A^\Phi$ as in \eqref{eq:rep-Aphi}.  However, written in this way we cannot apply any of our results since $\partial X=\rL^2(\partial\Omega)$ is neither an AM- nor an AL-space.

\smallbreak
To overcome this difficulty observe that $\phi\in \rC(\partial\Omega,\overline{\Omega})$ is uniformly continuous which implies that 
\[\rg(\Phi)\subset \rC(\partial\Omega)=:\widehat{\partial X}.\] 
Hence, we can define $\widehat \Phi:X\to
\widehat{\partial X}$, $\widehat \Phi f:=\Phi f$ for $f\in X$ which is bounded by the closed graph theorem. Moreover, let
$\widehat{L_\lambda}:={L_\lambda}|_{\widehat{\partial X}}\in\sL(\widehat{\partial X},X)$ and $\widehat{L_A}:=-A_{-1}\widehat{L_0}:\widehat{\partial X}\to X_{-1}$. Then clearly we also have
\[
A^\Phi=\bigl(A_{-1}+\widehat{L_A}\cdot\widehat\Phi\bigr)
\]
where $\widehat\Phi\in\sL(X,\widehat{\partial X})$ and $\widehat{L_A}\in\sL(\widehat{\partial X},\Xmo)$. Since $\widehat{\partial X}$ equipped with the sup-norm $\|\cdot\|_\infty$ is an AM-space we can now apply \Cref{cor:dX-AM}. To this end we first assume in addition that $\phi\ge0$ which obviously implies $\widehat{\Phi}\ge0$. Then 
\begin{equation}\label{eq:est-r(PLl)}
\sr\bigl(\widehat{\Phi}\widehat{L_\lambda}\bigr)\le\bigl\|\widehat{\Phi}\widehat{L_\lambda}\bigr\|_{\sL(\widehat{\partial X})}
\le\bigl\|\widehat{\Phi}\bigr\|_{\sL(X,\widehat{\partial X})}\cdot\bigl\|\widehat{L_\lambda}\bigr\|_{\sL(\widehat{\partial X},X)}.
\end{equation}
To estimate the norm of $\widehat{L_\lambda}$ observe that this operator is positive and for the $\|\cdot\|_\infty$-norm on $\widehat{\partial X}=\rC(\partial\Omega)$ it holds $|y|\le\one$ for all $y\in\widehat{\partial X}$ satisfying $\|y\|_{\widehat{\partial X}}\le1$, where $\one(z):=1$ for all $z\in\partial\Omega$. Hence, using \eqref{eq:def-L_mu}  and \cite[Lem.II.3.4.(i)]{EN:00} we conclude 
\begin{align*}
\bigl\|\widehat{L_\lambda}\bigr\|_{\sL(\widehat{\partial X},X)}
&=\sup\bigl\{\bigl\|\widehat{L_\lambda} y\bigr\|_X:y\in\widehat{\partial X},\|y\|_{\widehat{\partial X}}\le1\bigr\}\\
&=\|L_\lambda\one\|_X=\|-AR(\lambda,A)L_0\one\|_X\\
&=\bigl\|(Id-\lambda R(\lambda,A)\bigr)L_0\one\bigl\|_X\to0
\end{align*}
as $\lambda\to+\infty$. Therefore, by \eqref{eq:est-r(PLl)} there exists $\lambda>0$ such that $\sr(\widehat{\Phi}\widehat{L_\lambda})<1$ and by \Cref{cor:dX-AM} we conclude that $A^\Phi$ generates a positive semigroup on $X$.

\smallbreak
Now let $\phi\in \rC(\partial\Omega,\Omega)$ be arbitrary and denote the integral operator associated to $0\le|\phi|\in \rC(\partial\Omega,\Omega)$ by $\widetilde{\Phi}$. Then obviously $|\widehat{\Phi}y|\le\widetilde{\Phi}y$ for all $y\in\widehat{\partial X}$. Hence, by \Cref{thm:Dom} we obtain that  for arbitrary $\phi\in \rC(\partial\Omega,\Omega)$ the operator
\begin{align*}
&A^\Phi:=\Delta:D(A^\Phi)\subset \rL^2(\Omega)\to \rL^2(\Omega),\\
&D(A^\Phi):=\Bigl\{f\in \rH^{\frac12}(\Omega)\cap \rH^2_{\mathrm{loc}}(\Omega):\Delta f\in \rL^2(\Omega),\, f(z)=\int_\Omega\phi(z,x)\,f(x)\,dx,\, z\in\partial\Omega\Bigr\}
\end{align*}
generates a $C_0$-semigroup on $\rL^2(\Omega)$. Since \eqref{eq:HE-bfb} can be rewritten as \eqref{ACP-1} for $G:=A^\Phi$, we conclude that equation \eqref{eq:HE-bfb} is well-posed.

\begin{remark}
Since $X=\rL^2(\Omega)$ is not an AM-space, the before-mentioned results from \cite{BJVW:18} cannot be applied to this example. Hence, imposing that the observation/control space $U$ is of type AM (or AL) is a real generalization of the hypothesis that the state space $X$ itself has this property. Also, the results of \cite{Gre:87} don't work in this example since they are restricted to non-reflexive Banach spaces $X$.
\end{remark}
\subsection{Unbounded Perturbation of the First Derivative on $\mathbf{C_0(0,1]}$}\label{subsec:uP-FD}
We give an application of \Cref{AL,thm:Dom}. To this end we choose
\begin{itemize}
\item $X:=\rC_0(0,1]=\{f\in\rC[0,1]:f(0)=0\}$, $U=\rL^1[0,1]$ which is an AL-space,
\item $A:=-\frac d{dx}:D(A)\subset X\to X$, $D(A):=\{f\in\rC^1[0,1]:f(0)=f'(0)=0\}$,
\item $P:D(A)\subset X\to X$, $(P f)(x):=\int_0^x \frac{b(x-r)}{r^\alpha}\cdot f(r)\,dr$ for $\alpha\in[1, 2)$ and $b\in\rL^\infty[0,1]$, and
\item $G:=A+P:D(A)\subset X\to X$.
\end{itemize}
We claim that $G$ generates a $C_0$-semigroup which is positive in case $b\ge0$. 

\smallskip
For the proof we first observe that $A$ generates the positive nilpotent right-shift semigroup on $X$, cf.~\cite[Sect.I.4.17]{EN:00}. Then we factorize $P=BC$ where
\begin{itemize}
\item $B:U\to X$, $(B u)(x):=(b*u)(x)$, and
\item $C:D(A)\subset X\to U$, $(Cf)(x):=\frac{f(x)}{x^\alpha}$.
\end{itemize}
Indeed, $\rg(B)\subset X$ follows by Young's inequality while the fact $\alpha-1<1$ and
\[
\lim_{x\to0^+}\frac{\frac{f(x)}{x^\alpha}}{\;\frac{1}{x^{\alpha-1}}\;}=\lim_{x\to0^+}\frac{f(x)}{x}=f'(x)=0
\]
for $f\in D(A)$ implies $\rg(C)\subset U$. 

\smallskip
In order to apply \Cref{AL} we first assume that $b\ge0$ which implies $B,C\ge 0$. Then to verify that $\sr(CR(\lambda,A)B)<1$ for $\lambda>0$ sufficiently we estimate for $g\in U$
\begin{equation}\label{eq:est-CRB}
\bigl\|CR(\lambda,A)Bu\bigr\|_1
\le\bigl\|CR(\lambda,A)\bigr\|_{\sL(C_0,L^1)}\cdot\bigl\|Bu\bigr\|_\infty.
\end{equation}
Now a simple computation shows that for $f\in X$ and $\lambda>0$ we have
\[
\bigl(R(\lambda,A)f\bigr)(x)=\int_0^x e^{-\lambda(x-r)}f(r)\,dr,\ x\in[0,1].
\]
Hence, for $f\in X$ satisfying $\|f\|_\infty\le1$ we obtain
\begin{equation}\label{eq:est-CRf}
\begin{aligned}
\bigl\|CR(\lambda,A)f\bigr\|_1&=\int_0^1\Bigl|x^{-\alpha}\int_0^x e^{-\lambda(x-r)}f(r)\,dr\Bigr|\,dx\\
&\le\int_0^1x^{-\alpha}\int_0^x e^{-\lambda(x-r)}\,dr\,dx\\
&= \int_0^1\lambda^{-1}x^{-\alpha}\bigl(1-e^{-\lambda x}\bigr)\,dx\\
&=:\int_0^1r_\lambda(x)\,dx.
\end{aligned}
\end{equation}

Since $1-e^{-\lambda x}\le \lambda x$ for all $x\in[0,1]$ and $\lambda>0$ it follows that $r_\lambda(x)\le x^{1-\alpha}=:m(x)$ for all $x\in[0,1]$. Moreover,  $m\in\rL^1[0,1]$ since $\alpha<2$ and $\lim_{\lambda\to+\infty}r_\lambda(x)=0$ for all $x\in(0,1]$. By Lebesgue's dominated convergence theorem it follows that
\[
\int_0^1r_\lambda(x)\,dx\to0\quad\text{as $\lambda\to+\infty$}.
\]
By \eqref{eq:est-CRf} this implies $\|CR(\lambda,A)\|_{\sL(C_0,L^1)}\to0$ as $\lambda\to+\infty$ and from \eqref{eq:est-CRB} we finally obtain 
\[
\sr\bigl(CR(\lambda,A)B\bigr)\le\bigl\|CR(\lambda,A)B\bigr\|_{\sL(U)}\to0\quad\text{as $\lambda\to+\infty$}.
\]
At this point we can apply \Cref{AL} and conclude that $G$ is a generator if $b\ge0$. In case $b\in\rL^\infty[0,1]$ is not positive, we define the positive convolution operators $B_\pm\in\sL(U,X)$ by $(B_\pm u)(x):=(b_\pm*u)$ for which $B=B_+-B_-$ holds. The assertion then follows from \Cref{thm:Dom}.

\begin{remark}
In this example we applied \Cref{AL} regarding AL-spaces to the perturbation of an operator $A$ acting on an AM-space. We note that the results of \cite{BJVW:18} on AM-spaces are not applicable in this case since they only cover perturbations $P:X\to\Xmo$. Moreover, the results of \cite{Des:88,Voi:89} are not applicable since $X$ is not an AL-space.
Finally, we mention that by essentially the same proof this example can be generalized substituting the convolution operator $B$ by an arbitrary regular operator $B:U\to X$.
\end{remark}
\subsection{Rank-One Perturbation of the First Derivative on $\mathbf{L^p[0,1]}$}\label{subsec:r1p-FD}

We give an application of \Cref{PT,thm:Dom}. To this end we consider the partial differential equation
\begin{equation}\label{eq:ACP-1}
\begin{cases}
\begin{aligned}
\tfrac\partial{\partial t}u(t,x)&=\tfrac\partial{\partial x}u(t,x)+\Phi\bigl(u(t,\cdot)\bigl)\cdot b(x),\quad &&x\in[0,1],\ t\ge0,\\
u(0,x)&=u_0(x),\ u(t,1)=0,\quad &&x\in[0,1],\ t\ge0,
\end{aligned}
\end{cases}
\end{equation}
for some linear functional $\Phi\in \rC[0,1]'$ and a function $b\in \rL^1[0,1]$. In \cite[Expl.5.1]{BJVW:18} it was shown that for  $\Phi(f)=\int_0^1 f(x)\,dx$ this problem is well-posed on the AM-space $X=C_0[0,1)$. This would also follow easily from \Cref{AM} and the subsequent computations. However, here we consider \eqref{eq:ACP-1}  on $X=\rL^p[0,1]$ for $1\le p<+\infty$ showing well-posedness for arbitrary $\Phi\in \rC[0,1]'$.

\smallbreak
In order to apply our results we introduce the following spaces and operators:
\begin{itemize}
\item $X^p:=\rL^p[0,1]$ for some $1\le p<+\infty$, $U=\RR$,
\item $A^p:=\frac d{dx}:D(A^p)\subset X^p\to X^p$ for $D(A^p):=\{f\in \rW^{1,p}[0,1]:f(1)=0\}$,
\item $B:U\to \rL^1[0,1]$, $B\alpha:=\alpha\cdot b$ for all $\alpha\in U$,
\item $C:=\Phi:Z\subset X\to U$ for $Z:=\rC[0,1]$.
\end{itemize}
It is well known that $A^p$ has empty spectrum $\sigma(A^p)$ and generates the positive nilpotent left-shift semigroup $(T^p(t))_{t\ge0}$ given by 
\begin{equation*} 
\bigl(T^p(t)f\bigr)(x):=
\begin{cases}
f(x+t)&\text{if $x+t\le1$},\\
0&\text{else.}
\end{cases}
\end{equation*}
Now fix $p\ge1$ and define 
\begin{itemize}
\item $X:=X^p$, $A:=A^p$, $T(t):=T^p(t)$ for all $t\ge0$, and
\item $\Xt:=X^1$, $\At:=A^1$, $\Tt(t):=T^1(t)$ for all $t\ge0$.
\end{itemize}
Since $\rW^{1,1}[0,1]\hookrightarrow \rL^p[0,1]\hookrightarrow \rL^1[0,1]$ with continuous injections, \Cref{lem:sub-sgr}.(3) implies $\Xt\subseteq X_{-1}$ and $B\in\sL(U,X_{-1})$. Moreover, by (4) \& (5) the extrapolated operators $R(\lambda,A_{-1})$, $T_{-1}(t)$ and $A_{-1}$ restricted to $\Xt$ coincide with $R(\lambda,\At)$, $\Tt(t)$ and $\At$, respectively. 

\smallskip
We proceed by verifying the assumptions of \Cref{PT} in case $\Phi\in \rC[0,1]'_+$ and $b\in \rL^1[0,1]_+$ are positive. Clearly this implies that also $C:Z\to U$ and $B:U\to\Xmo$ are positive. Since $\rg(R(\lambda,A_{-1})B)=\spn(R(\lambda,\At)b)\in \rW^{1,1}[0,1]\subset \rC[0,1]=Z$ the triple $(A,B,C)$ is compatible. 
In order to verify that $B$ is $p$-admissible we compute for $u\in \rL^p[0,1]$
\begin{align*}
g(\cdot):&=\int_0^1T_{-1}(1-s)Bu(s)\,ds=\int_0^1\Tt(1-s)b(\cdot)u(s)\,ds\\
&=\int_{\bf\cdot}^1b(\cdot+1-s)u(s)\,ds=(b_1*\tilde u)(\cdot),
\end{align*}
where $\tilde u$ is the extension of $u:[0,1]\to\RR$ to all of $\RR$ by the value $0$ and
\[
b_1(x):=
\begin{cases}
b(x+1)&\text{if $x\in[-1,0]$},\\
0&\text{if $x\in\RR\setminus[-1,0]$}.
\end{cases}
\]
Since $b_1\in \rL^1(\RR)$ and $\tilde u\in \rL^p(\RR)$, by Young's inequality it follows $b_1*\tilde u\in \rL^p(\RR)$. This implies $g=(b_1*\tilde u)|_{[0,1]}\in \rL^p[0,1]$, hence $B$ is $p$-admissible. 

\smallbreak
The $p$-admissibility of $C=\Phi\in \rC[0,1]'$ is shown in \cite[proof of Cor.25, (iii)]{ABE:14}. It only remains to show that $\sr(CR(\lambda,A_{-1})B)=\Phi R(\lambda,\At)b<1$ for some $\lambda\in\RR$.  Since $\At$ generates a $C_0$-semigroup on $\Xt$ there exists $M\ge1$ such that 
\[
\bigl\|R(\lambda,\At)\bigr\|_{\sL(\Xt,\Xt_1)}\le M\quad\text{for all $\lambda\ge0$}.
\]
Moreover, for $f\in D(\At)$ we have
\begin{align*}
\bigl\|R(\lambda,\At)f\bigr\|_{\Xt_1}
&=\bigl\|\At^{-1}R(\lambda,\At)\At f\bigr\|_{\Xt_1}\\
&\le\bigl\|\At^{-1}\bigr\|_{\sL(\Xt,\Xt_1)}\cdot\bigl\|R(\lambda,\At)\At f\bigr\|_{\Xt}
\to 0\quad\text{as $\lambda\to+\infty$}.
\end{align*}
By \cite[Sec.III.4.5]{Sch:80} this implies $\|R(\lambda,\At)f\|_{\Xt_1}\to0$ as $\lambda\to+\infty$ for all $f\in\Xt=\rL^1[0,1]$.
Since $\Xt_1=\rW^{1,1}[0,1]\hookrightarrow\rC[0,1]=Z$ with continuous injection we conclude
\[
\lim_{\lambda\to+\infty}C R(\lambda,A_{-1})B
\le\lim_{\lambda\to+\infty}\|\Phi\|_{Z'}\cdot\bigl\|R(\lambda,\At)b\bigr\|_Z=0.
\]
Summing up, by \Cref{PT} we obtain that for positive $C=\Phi\in Z'_+$ and $b\in\Xt_+$ the operator $A_{BC}$ generates a positive semigroup on $X$.

\smallbreak
If $\Phi\in Z'$ and $b\in\Xt$ are arbitrary we use the Jordan decomposition 
to write $\Phi=\Phi_+-\Phi_-$ for positive functionals $\Phi_\pm\in Z'_+$.  Then $\widetilde{C}:=\Phi_++\Phi_-$  satisfies \Cref{MA-dom}.(i). Moreover, for the operators $B_\pm:U\to X_{-1}$, $B\alpha:=\alpha\cdot b_\pm$, $\alpha\in U$ also \Cref{MA-dom}.(ii) is satisfied, hence by \Cref{thm:Dom} the operator 
\begin{align*}
&G:=\tfrac d{ds}+\Phi\otimes b:D(G)\subset \rL^p[0,1]\to \rL^p[0,1],\\
&D(G):=\bigl\{f\in \rW^{1,1}[0,1]: f'+\Phi(f)\cdot b\in \rL^p[0,1]\bigr\}
\end{align*}
generates a $C_0$-semigroup on $\rL^p[0,1]$ for arbitrary $1\le p<+\infty$, $\Phi\in \rC[0,1]'$ and $b\in \rL^1[0,1]$, where $\Phi\otimes b:Z\to X_{-1}$, $(\Phi\otimes b)f:=\Phi(f)\cdot b$. Since \eqref{eq:ACP-1} on $X$ is equivalent to \eqref{ACP-1}, this implies that the equation \eqref{eq:ACP-1} is well-posed.

\begin{remark}
We note that in case $p>1$ the space $X=\rL^p[0,1]$ is neither of type AM- nor AL, and for $b\in \rL^1[0,1]\setminus X\ne\emptyset$ we have $\rg(B)=\spn(b)\not\subset X$. 
In addition in our setting the operator $C\in \rC[0,1]'$ might be unbounded on $X$, which means that the perturbation $P=\Phi\otimes b$ changes both action and domain of the operator $A$. In particular, in this case the results from \cite{BJVW:18,Des:88,Voi:89} based on the Miyadera--Voigt or Desch--Schappacher perturbation theorems (where one of the operators $C$ or $B$ has to be bounded) are not applicable to this situation.
\end{remark}

\section{Conclusion}
In this paper we showed how factorizing a perturbation $P:Z\to X_{-1}$ as $P=BC$ for operators $C:Z\to U$ and $B:U\to X_{-1}$ via a Banach lattice $U$ as in the diagram
\begin{equation*} 
\begindc{\commdiag}[20]
\obj(10,20)[Z]{$D(A)\subseteq Z\subseteq X$}
\obj(60,20)[U]{$U$}
\obj(112,20)[X]{$X_{-1}$}
\cmor((14,24)(15,28)(19,29)(62,29)(105,29)(109,28)(110,24)) \pdown(60,33){$P$}
\cmor((14,16)(15,12)(19,11)(37,11)(54,11)(58,12)(59,16)) \pup(35,7){$C$}
\cmor((61,16)(62,12)(66,11)(85,11)(105,11)(109,12)(110,16)) \pup(85,7){$B$}
\enddc
\end{equation*}
allows to generalize significantly previous results on the perturbation of the generator $A$ of a positive semigroup on a Banach lattice $X$. 
More precisely, in \cite{Des:88,Voi:89,BJVW:18} the state space $X$ itself has to be an AM- or AL-space which, e.g., a priori excludes applications on (infinite dimensional) reflexive spaces. In contrast, our approach only needs $U$ to be of type AM or AL which, as shown in \Cref{sec:App}, significantly widens the possible applications. Moreover, if $U=\R$ we can even treat unbounded perturbations $P=BC:Z\subset X\to X_{-1}$ with $\rg(P)\not\subset X$ which also change the domain of the unperturbed operator $A$. 
\appendix

\section{Appendix}

\subsection{Banach Lattices and Positive Operators}\label{appa}
In this appendix we briefly recall some basic facts about Banach lattices and positive operators. Fore more details we refer to \cite[Chap.C-I]{Nag(e):86}, \cite[Chap.10]{BKR:17} or \cite[Sect.2.2]{BA:06}. Moreover, we state and proof a Riesz-Thorin type theorem for $\R$-valued functions.

\smallbreak
A real vector space $X$ equipped with a partial order $\le$ is a \emph{vector lattice} if for any pair $x,y\in X$  of elements the greatest lower bound $\inf\{x,y\}\in X$ and the least upper bound $\sup\{x,y\}\in X$ exist, and the order is compatible with the vector space structure, i.e.,
\begin{itemize}
\item $x\le y$ implies $x+z\le y+z$ for all $x,y,z\in X$;
\item $0\le x$ implies $0\le \alpha x$ for all $\alpha\in\mathbb{R}^+$.
\end{itemize}
For a vector lattice $X$ we denote by 
\[X_+:=\{x\in X:\,0\le x\}\]
its \emph{positive cone}. Moreover, for $x\in X$ we define its \emph{positive part} $x^+=\sup\{x,0\}$, its \emph{negative part} $x^-=\sup\{-x,0\}$ and its \emph{absolute value} $|x|=\sup\{x,-x\}$. In this case we have $x=x^+-x^-$ and $|x|=x^++x^ -$.

\smallbreak
A norm $\lVert\cdot\rVert$ on a vector lattice $X$ is called a \emph{lattice norm} if it is compatible with the order, i.e., if
\begin{equation}\label{lat}
|x|\le|y|\Rightarrow\lVert x\rVert\le\lVert y\rVert\quad\text{for }x,y\in X.
\end{equation}
Note that for a lattice norm we always have $\||x|\|=\|x\|$ for all $x\in X$.
If a real vector lattice $X$ endowed with a lattice norm is complete, then $X$ is called a real \emph{Banach lattice}. 
A Banach lattice $X$ is said to be an
\begin{itemize}
\item \emph{AL-space}, if $\|x+y\|=\|x\|+\|y\|$ for all $x,y\ge0$,
\item \emph{AM-space}, if $\|\sup\{x,y\}\|=\sup\{\|x\|,\|y\|\}$ for all $x,y\ge0$.
\end{itemize}
We remark that if for some Banach lattice $X$ there exist two comparable norms $\|\cdot\|_1$ and $\|\cdot\|_\infty$ such that $(X\|\cdot\|_1)$ is of type AL and $(X\|\cdot\|_\infty)$ of type AM, then $X\simeq\R$ for some $N\in\NN$. 

\smallbreak
A vector subspace $\mathcal{C}$ of a vector lattice $X$ is called \emph{vector sublattice} if for all $x\in\mathcal{C}$ one has $|x|\in\mathcal{C}$, hence also $x^+,x^-\in\mathcal{C}$.

\begin{remark}\label{posden}
For a Banach lattice $X$, the operations $\sup\{\cdot,\cdot\}$ and $\inf\{\cdot,\cdot\}$ are continuous. This implies that the positive cone $X_+$ is closed (cf.\  \cite[Prop.10.8]{BKR:17}).
\end{remark}

Besides $\R$ the standard examples of Banach lattices are the spaces $l^p$, $\rC(K)$, $\rC_0(\Omega)$ and $\rL^p(\Omega;\mu)$ equipped with the natural order and the canonical norms. In particular
\begin{itemize}
\item $(\R,|\cdot|_1)$, $l^1$ and $\rL^1(\Omega;\mu)$ are AL-spaces,
\item $(\R,|\cdot|_\infty)$,  $c$, $c_0$, $l^\infty$, $\rL^\infty(\Omega;\mu)$, $\rC(K)$ and $\rC_0(\Omega)$ are AM-spaces,
\end{itemize}
see \cite[Expl.10.6]{BKR:17}.

\begin{definition}\label{posop}
If $X$ and $Y$ are two real Banach lattices, an operator $T:X\to Y$ is said to be \emph{positive} if $Tx\in Y_+$ for every $x\in X_+$. In this case we use the notation $T\ge0$. Operators which can be written as the difference of two positive operators are called \emph{regular}.
\end{definition}
We remark that any linear positive operator $T:X\to Y$ between Banach lattices is bounded, see \cite[Thm.10.20]{BKR:17}.
An important characterization of positive operators is the following.

\begin{proposition}\label{pos}\cite[Lem.10.18]{BKR:17}
For a linear operator $T:X\to Y$ between two real Banach lattices, the following are equivalent:
\begin{enumerate}[label=(\alph*)]
	\item $T\ge0$;
	\item $(Tx)^+\le Tx^+$ and $(Tx)^-\le Tx^-$ for all $x\in X$;
	\item $|Tx|\le T|x|$ for all $x\in X$.
\end{enumerate}
\end{proposition}

We denote by $\sL(X,Y)_+$ the set of positive linear operators from the Banach lattice $X$ to the Banach lattice $Y$. For positive operators the operator norm is given by
\begin{equation*} 
\lVert T\rVert_{\sL(X,Y)}:=\sup\bigl\{\lVert Tx\rVert_Y: x\in X_+\,\,\text{and}\,\,\lVert x\rVert_X\le 1\bigr\}.
\end{equation*}
When there is no risk of ambiguity the operator norm is simply denoted by $\lVert\cdot\rVert$.

Since the norm $\lVert\cdot\rVert_Y$ on the Banach lattice $Y$ satisfies \eqref{lat}, the operator norm inherits the same property, i.e., if $S,T\in\sL(X,Y)_+$ satisfy $S\le T$, then $\lVert S\rVert\le\lVert T\rVert$.

We recall that for a bounded operator $T$ on a Banach space $X$ its \emph{spectrum} 
\[\sigma(T):=\bigl\{\lambda\in\mathbb{C}:\lambda-T\text{ is not bijective}\bigr\}\] 
is a nonempty, compact subset of $\mathbb{C}$. Consequently, the \emph{resolvent set} $\rho(A):=\mathbb{C}\setminus\sigma(A)$ is an open, nonempty subset of $\mathbb{C}$. 
The \emph{spectral radius} of $T$ is defined as 
\begin{equation*}
\sr(T):=\sup\bigl\{|\lambda|:\lambda\in\sigma(T)\bigr\}
\end{equation*}
which satisfies $\sr(T)\le\lVert T\rVert$ (see \cite[Cor.IV.1.4]{EN:00}).

\smallbreak
The spectral radius can be calculated using \emph{Gelfand's formula}.
\begin{proposition}\cite[Prop.3.3]{BKR:17}\label{gelf}
For an operator $T\in\sL(X)$ on a Banach space $X$
	\begin{equation}\label{gelfor}
		\sr(T)=\lim_{n\to\infty}\lVert T^n\rVert^{1/n}.
	\end{equation}
Moreover, if $|\lambda|>\sr(T)$ then $R(\lambda,T)$ is given by the \emph{Neumann series}
\[
R(\lambda,T)=\sum_{n=0}^{+\infty}\lambda^{-(n+1)}\cdot T^n.
\]
\end{proposition}

The identity \eqref{gelfor} implies the monotonicity of the spectral radius as follows.

\begin{proposition}\label{spr}
Let $S,T:X\to X$ be linear operators on a Banach lattice $X$ satisfying $|Sx|\le Tx$ for every $x\ge0$. Then 
\begin{enumerate}[label=(\roman*)]
\item $|S^nx|\le T^n|x|$ for every $n\in\NN$ and $x\in X$;
\item $S,T\in\sL(X)$ and $\|S^n\|\le\|T^n\|$ for every $n\in\NN$;
\item $\sr(S)\le \sr(T)$.
\end{enumerate}
\end{proposition}

\begin{proof}
We start by proving (i) by induction. For $n=1$ we have for $x=x^+-x^-\in X$
\begin{equation*}
|Sx|=\bigl|S(x^+-x^-)\bigr|\le|Sx^+|+|Sx^-|\le T(x^++x^-)=T|x|
\end{equation*}
as claimed. Now assume that (i) holds for some $n\in\NN$. Then
\begin{equation*}
\bigl|S^{n+1}x\bigr|=\bigl|S^n Sx\bigr|\le T^n|Sx|\le T^n T|x|=T^{n+1}|x|
\end{equation*}
which proves (i). To prove (ii) we first observe that by \cite[Thm.10.20]{BKR:17} the operator $T$ is bounded. Now take $x\in X$. Then from (i) we obtain
\begin{equation*}
\|S^nx\|=\bigl\||S^nx|\bigr\|\le\|T^n|x|\|\le\|T^n\|\cdot\||x|\|=\|T^n\|\cdot\|x\|.
\end{equation*}
This implies (ii). Finally, using (ii) we then conclude that
\[\sr(S)=\lim_{n\to\infty}\lVert S^n\rVert^{1/n}\le\lim_{n\to\infty}\lVert T^n\rVert^{1/n}=\sr(T)\]
as claimed.
\end{proof}

Finally, for the generator $A$ of a $C_0$-semigroup $(T(t))_{t\ge0}$ we introduce its \emph{spectral bound} $\spb(A)$ and \emph{growth bound} $\omega_0(A)$ by
\begin{align*}
\spb(A)&:=\sup\bigl\{\Re(\lambda):\lambda\in\sigma(A)\bigr\},\\
\omega_0(A)&:=\inf\{\omega\in\mathbb{R}:\exists M_\omega\ge1\text{ such that }\bigl\|T(t)\bigr\|\le M_\omega\cdot e^{\omega t}\ \forall t\ge0\bigr\}.
\end{align*}
Then by \cite[Cor.II.1.13]{EN:00} we always have $-\infty\le \spb(A)\le\omega_0(A)<+\infty$.

\smallbreak
The following result is a version of the Riesz--Thorin convexity theorem for positive operators defined on spaces of vector-valued functions. The proof of this theorem is inspired by \cite{Haa:07}: we first give a Hölder-type inequality for these operators and then use the latter to provide an estimate on their norms. 

\begin{theorem}\label{RT}
Let $T_0\in\sL(\ceor)$ and $T_1\in\sL(\lrr)$ be positive operators such that $\|T_0\|_{\sL(C_0)}\le M_0$, $\|T_1\|_{\sL(\rL^1)}\le M_1$ and
$T_0f=T_1f$ for all $f\in\mathcal{C}:=C_c((0,+\infty),\R)$.
Then for any $1<p<\infty$ the  operator 
\[T:\mathcal{C}\subset\lprr\to\lrr\cap\ceor\subset\lprr\]
given by $Tf:=T_0f=T_1f$ for $f\in\mathcal{C}$, possesses a unique positive bounded extension (still denoted by $T$) $T\in\sL(\lprr)$ satisfying
\begin{equation}\label{rito}
	\lVert T\rVert_{\sL(\rL^p)}\le M_0^{1-1/p}\cdot M_1^{1/p}.
\end{equation}
\end{theorem}

\begin{proof}
By elementary calculus one can easily verify that for real numbers $a,b\ge0$ and all $1\le p\le\infty$ one has
\begin{equation}\label{you}
	ab=\inf_{t>0}\biggl(\dfrac{t^p}{p}\,a^p+\dfrac{t^{-p'}}{p'}\,b^{p'}\biggr),
\end{equation}
where $p'$ is the Hölder conjugate of $p$, i.e., $\frac{1}{p}+\frac{1}{p'}=1$.
Using this fact we next prove a Hölder-type inequality for positive operators. 

\smallbreak
To this end we take $0\le f,g\in\mathcal{C}$. If $f=(f_1,\dots,f_N)^T$ and $g=(g_1,\dots,g_N)^T$, we define the $p$-th power of $f$ and the product of $f g$ componentwise as
$$f^p:=
\left(\begin{matrix}
	f_1^p\\\vdots\\f_N^p
\end{matrix}\right)
\qquad\text{and}\qquad
f g:=
\left(\begin{matrix}
	f_1g_1\\\vdots\\f_Ng_N
\end{matrix}\right).$$
Since \eqref{you} holds for any component of $(fg)(s)$ for any $s\ge0$ we conclude
\begin{equation}\label{mat}
	f\cdot g=\left(\begin{matrix}
		f_1g_1\\\vdots\\f_Ng_N
	\end{matrix}\right)\leq\dfrac{t^p}{p}\left(\begin{matrix}
		f_1^p\\\vdots\\f_N^p
	\end{matrix}\right)+\dfrac{t^{-p'}}{p'}\left(\begin{matrix}
		g_1^{p'}\\\vdots\\g_N^{p'}
	\end{matrix}\right)=\dfrac{t^p}{p}f^p+\dfrac{t^{-p'}}{p'}g^{p'}\quad\text{for all $t>0$.}
\end{equation}

Since also $f^p,g^{p'}, fg\in\mathcal{C}_+$, we can apply the positive operator $T$ to \eqref{mat} and obtain
\[T(fg)\le\dfrac{t^p}{p}\,Tf^p+\dfrac{t^{-p'}}{p'}\,Tg^{p'},\]
i.e., 
\[\bigl[T(fg)\bigr](s)\le\dfrac{t^p}{p}\bigl[Tf^p\bigr](s)+\dfrac{t^{-p'}}{p'}\bigl[Tg^{p'}\bigr](s)\quad\text{for all $t>0$ and $s\ge0$.}\]
Passing, for fixed $s\ge0$, component-wise to the infimum over all $t>0$, by \eqref{you} it follows that for any $i=1,\dots,N$
\begin{equation*}
	\bigl[T(fg)\bigr]_i(s)\le\bigl[Tf^p\bigr]_i^{1/p}(s)\cdot\bigl[Tg^{p'}\bigr]_i^{1/p'}(s)\quad\text{for all $s\ge0$,}
\end{equation*}
i.e., we obtain the Hölder-type inequality
\begin{equation}\label{you3}
	T(fg)\le\bigl(Tf^p\bigr)^{1/p}\cdot\bigl(Tg^{p'}\bigr)^{1/p'}\quad\text{for all $f,g\in\mathcal{C}_+$. }
\end{equation}

To conclude the proof, for $f\in\mathcal{C}_+$ we choose $0<a<b<+\infty$ such that $\text{supp}(f)\subset[a,b]$. Then there exists a function $\mathbbm{1}_f:[0,+\infty)\to\mathbb{R}^N$ satisfying
\begin{itemize}
	\item $\mathbbm{1}_f\in\mathcal{C}$;
	\item $0\le\mathbbm{1}_f(s)\le1$ for all $s\ge0$;
	\item $\mathbbm{1}_f(s)=1$ for all $s\in[a,b]$;
	\item $\mathbbm{1}_f(s)=0$ for all $s\in[0,\frac{a}{2}]\cap[b+1,+\infty)$.
\end{itemize}
Hence, $f\cdot\mathbbm{1}_f=f$. Since $f,\mathbbm{1}_f\in\mathcal{C}_+$, we can take $g=\mathbbm{1}_f$ in \eqref{you3} and obtain that
\[\bigl(T(f\cdot\mathbbm{1}_f)\bigr)^p=(Tf)^p\le\bigl(Tf^p\bigr)\cdot\bigl(T\mathbbm{1}_f^{p'}\bigr)^{p/p'},\]
which implies that for any $s\ge0$ we have
\begin{equation*}
	\begin{split}
\bigl\vert(Tf)(s)\bigr\rvert_p^p
&=\bigl\lvert(Tf)^p(s)\bigr\rvert_1
  \le\bigl\lvert(Tf^p)(s)\cdot(T\mathbbm{1}_f^{p'})^{p/p'}(s)\bigr\rvert_1\\
&\le\bigl\lvert(Tf^p)(s)\bigr\rvert_1\cdot\bigl\lvert(T\mathbbm{1}_f^{p'})^{p/p'}(s)\bigr\rvert_\infty,
	\end{split}
\end{equation*}
where as usual $|\cdot|_p$ denotes the $p$-norm on $\mathbb{R}^N$. Integrating for $s\ge0$ we conclude
\begin{equation*}
	\begin{split}
\lVert Tf\rVert_p^p
&=\int_{0}^{+\infty}\bigl\lvert(Tf)(s)\bigr\rvert_p^p\,ds
\le\int_{0}^{+\infty}\bigl\lvert(Tf^p)(s)\bigr\rvert_1\cdot\bigl\lvert(T\mathbbm{1}_f^{p'})^{p/p'}(s)\bigr\rvert_\infty\,ds\\
&\le\bigl\lVert Tf^p\bigr\rVert_{1}\cdot\bigl\lVert (T\mathbbm{1}_f^{p'})^{p/p'}\bigr\rVert_\infty=\bigl\lVert Tf^p\bigr\rVert_{1}\cdot\bigl\lVert T\mathbbm{1}_f^{p'}\bigr\rVert_\infty^{p-1}\\
& 
 \le M_1\cdot\lVert f\rVert_p^p\cdot M_0^{p-1}\bigl\lVert\mathbbm{1}_f^{p'}\bigr\rVert_\infty^{p-1}
 \le M_1\cdot M_0^{p-1}\cdot\lVert f\rVert_p^p,
	\end{split}
\end{equation*}
where in the last inequality we have used the fact that $\lVert\mathbbm{1}_f^{p'}\rVert_\infty\le1$. Therefore, for any $f\in\mathcal{C}_+$ we proved that
\begin{equation}\label{riesz}
	\lVert Tf\rVert_p\le M_0^{1-1/p}\cdot M_1^{1/p}\cdot\lVert f\rVert_p.
\end{equation}
Now take $f\in\mathcal{C}$. Then $|f|\in\mathcal{C}_+$, hence \eqref{riesz} implies
\begin{equation}\label{riesz2}
	\lVert Tf\rVert_p\le\bigl\lVert T|f|\bigr\rVert_p\le M_0^{1-1/p}\cdot M_1^{1/p}\cdot\lVert f\rVert_p.
\end{equation}
Since $\mathcal{C}$ is dense in $\lprr$ for all $1\le p<\infty$, the operator $T:\mathcal{C}\subset\lprr\to\lprr$ extends (uniquely) to a linear positive operator $T:\lprr\to\lprr$ still satisfying \eqref{riesz2}, hence \eqref{rito}.
\end{proof}

\subsection{Subspace Semigroups}
The following abstract result is needed in \Cref{subsec:r1p-FD}.

\begin{lemma}\label{lem:sub-sgr}
Let $\At:\Xt\to\Xt$ be the generator of a $C_0$-semigroup $(\Tt(t))_{t\ge0}$ on a Banach space $\Xt$. Suppose that $X$ is a Banach space such that
\begin{enumerate}[label=(\roman*)]
\item $\Xt_1\hookrightarrow X\hookrightarrow\Xt$ with continuous injections, where $\Xt_1:=(D(\At),\|\cdot\|_{\At})$,
\item $\Tt(t)X\subseteq X$ for all $t\ge0$,
\item $(T(t))_{t\ge0}:=(\Tt(t)|_X)_{t\ge0}$ defines a $C_0$-semigroup on $X$.
\end{enumerate}
If $A:D(A)\subset X\to X$ denotes the generator of $(T(t))_{t\ge0}$ then the following holds:
\begin{enumerate} 
\item $\rho(A)=\rho(\At)$ and $R(\lambda,\At)|_X=R(\lambda,A)$ for all $\lambda\in\rho(A)$.
\item For $\lambda_0\in\rho(\At)$ there exists $M>0$ such that $\|R(\lambda_0,\At)x\|_X\le M\cdot\|x\|_{\Xt}$ for all $x\in\Xt$.
\item $\Xt\subseteq X_{-1}$ with continuous injection.
\item $R(\lambda,A_{-1})|_{\Xt}=R(\lambda,\At)$ for all $\lambda\in\rho(A)$ and $T_{-1}(t)|_{\Xt}=\Tt(t)$ for all $t\ge0$.
\item $A_{-1}|_{\Xt}=\At$.
\end{enumerate}
\end{lemma}

\begin{proof} (1) is shown in \cite[Prop.IV.2.17]{EN:00}. Assertion~(2) follows by the continuity of $R(\lambda_0,\At):\Xt\to\Xt_1$ and the continuous injection of $\Xt_1$ into $X$. To show (3) we first note that by (1) and (2) the norm $\|\cdot\|_{\Xt}$ restricted to $X$ is finer than the extrapolated norm $\|\cdot\|_{-1}:=\|R(\lambda_0,A)\cdot\|_X$ on $X$. 
For the relative completions of $X$ with respect to these norms this implies
\begin{equation*}
\widetilde X=\bigl(\Xt_1,\|\cdot\|_{\Xt}\bigr)^{\sim}\subseteq\bigl(X,\|\cdot\|_{\Xt}\bigr)^{\sim}\subseteq\bigl(X,\|\cdot\|_{-1}\bigr)^\sim=X_{-1},
\end{equation*}
where for the first equation we used that $\Xt_1$ is dense in $\Xt$ with respect to $\|\cdot\|_{\Xt}$. This proves the third claim. In order to show (4) we define for $x\in\Xt$ the sequence $x_n:=nR(n,\At)x\in\Xt_1\subseteq X$, $n\in\NN$. Then as $n\to+\infty$ by \cite[Lem.II.3.4.(i)]{EN:00} we have $x_n\to x$ in $\Xt$, hence by (3) also in $X_{-1}$. By (1) this implies for all $\lambda\in\rho(A)$
\begin{align*}
R(\lambda,A_{-1})x&=\lim_{n\to+\infty}R(\lambda,A)x_n=\lim_{n\to+\infty}R(\lambda,\At)\cdot nR(n,\At)x\\
&=\lim_{n\to+\infty}n R(n,\At)\cdot R(\lambda,\At)x=R(\lambda,\At)x,
\end{align*}
where the last equality follows again by \cite[Lem.II.3.4.(i)]{EN:00}. Moreover, by the same reasoning we obtain for all $t\ge0$
\begin{align*}
T_{-1}(t)x&=\lim_{n\to+\infty}T(t)x_n=\lim_{n\to+\infty}\Tt(t)\cdot nR(n,\At)x\\
&=\lim_{n\to+\infty}n R(n,\At)\cdot \Tt(t)x=\Tt(t)x.
\end{align*}
This proves (4). Finally, to verify (5) take $x\in D(\At)=\Xt_1\subseteq X=D(A_{-1})$. Then using (4) we conclude
\begin{equation*}
\At x=\lim_{t\to0^+}\tfrac{\Tt(t)x-x}{t}=\lim_{t\to0^+}\tfrac{T_{-1}(t)x-x}{t}=A_{-1}x
\end{equation*}
which implies (5) and completes the proof.
\end{proof}

\newcommand{\urlprefix}{}


\begin{thebibliography}{BJVW18}
\providecommand{\url}[1]{\texttt{#1}}

\bibitem[ABE14]{ABE:14}
M.~Adler, M.~Bombieri, and K.-J. Engel.
\newblock \emph{On perturbations of generators of {$C_0$}-semigroups}.
\newblock Abstr. Appl. Anal.  (2014), Art. ID 213020.
\newblock \urlprefix\url{https://dx.doi.org/10.1155/2014/213020}.

\bibitem[ABE17]{ABE:15}
M.~Adler, M.~Bombieri, and K.-J. Engel.
\newblock \emph{Perturbation of analytic semigroups and applications to partial
  differential equations}.
\newblock J. Evol. Equ. \textbf{17} (2017), 1183--1208.
\newblock \urlprefix\url{https://doi.org/10.1007/s00028-016-0377-8}.

\bibitem[ABHN11]{ABHN:11}
W.~Arendt, C.~J.~K. Batty, M.~Hieber, and F.~Neubrander.
\newblock \emph{Vector-{V}alued {L}aplace {T}ransforms and {C}auchy
  {P}roblems}, \emph{Monographs in Mathematics}, vol.~96.
\newblock Birkhäuser/Springer Basel AG, Basel, 2nd edn. (2011).
\newblock \urlprefix\url{https://dx.doi.org/10.1007/978-3-0348-0087-7}.

\bibitem[AE18]{AE:18}
M.~Adler and K.-J. Engel.
\newblock \emph{Spectral theory for structured perturbations of linear
  operators}.
\newblock J. Spectr. Theory \textbf{8} (2018), 1393--1442.
\newblock \urlprefix\url{https://doi.org/10.4171/JST/230}.

\bibitem[AR91]{AR:91}
W.~Arendt and A.~Rhandi.
\newblock \emph{Perturbation of positive semigroups}.
\newblock Arch. Math. \textbf{56} (1991), 107--119.
\newblock \urlprefix\url{https://doi.org/10.1007/BF01200341}.

\bibitem[Are87]{Are:87b}
W.~Arendt.
\newblock \emph{Resolvent positive operators}.
\newblock Proc. London Math. Soc. \textbf{54} (1987), 321--349.

\bibitem[BA06]{BA:06}
J.~Banasiak and L.~Arlotti.
\newblock \emph{Perturbations of {P}ositive {S}emigroups with {A}pplications}.
\newblock Springer Monographs in Mathematics. Springer-Verlag London, Ltd.,
  London (2006).
\newblock \urlprefix\url{https://doi.org/10.1007/1-84628-153-9}.

\bibitem[BE14]{BE:14}
M.~Bombieri and K.-J. Engel.
\newblock \emph{A semigroup characterization of well-posed linear control
  systems}.
\newblock Semigroup Forum \textbf{88} (2014), 366--396.
\newblock \urlprefix\url{https://doi.org/10.1007/s00233-013-9545-0}.

\bibitem[BJVW18]{BJVW:18}
A.~B\'{a}tkai, B.~Jacob, J.~Voigt, and J.~Wintermayr.
\newblock \emph{Perturbations of positive semigroups on {AM}-spaces}.
\newblock Semigroup Forum \textbf{96} (2018), 333--347.
\newblock \urlprefix\url{https://doi.org/10.1007/s00233-017-9879-0}.

\bibitem[BKFR17]{BKR:17}
A.~B\'{a}tkai, M.~Kramar~Fijav\v{z}, and A.~Rhandi.
\newblock \emph{Positive {O}perator {S}emigroups}, \emph{Operator Theory:
  Advances and Applications}, vol. 257.
\newblock Birkh\"{a}user/Springer, Cham (2017).
\newblock \urlprefix\url{https://doi.org/10.1007/978-3-319-42813-0}.

\bibitem[CENN03]{CENN:03}
V.~Casarino, K.-J. Engel, R.~Nagel, and G.~Nickel.
\newblock \emph{A semigroup approach to boundary feedback systems}.
\newblock Integral Equations Operator Theory \textbf{47} (2003), 289--306.
\newblock \urlprefix\url{https://dx.doi.org/10.1007/s00020-002-1163-2}.

\bibitem[Des88]{Des:88}
W.~Desch.
\newblock \emph{Perturbation of positive semigroups on {AL}-spaces} (1988).
\newblock (preprint 1988).

\bibitem[EG23]{EGa:23}
Y.~El~Gantouh.
\newblock \emph{Positivity of infinite-dimensional linear systems}.
\newblock arXiv  (2023).
\newblock \urlprefix\url{https://doi.org/10.48550/arXiv.2208.10617}.

\bibitem[EN00]{EN:00}
K.-J. Engel and R.~Nagel.
\newblock \emph{One-{P}arameter {S}emigroups for {L}inear {E}volution
  {E}quations}, \emph{Graduate Texts in Mathematics}, vol. 194.
\newblock Springer-Verlag, New York (2000).
\newblock \urlprefix\url{https://dx.doi.org/10.1007/b97696}.

\bibitem[Gre87]{Gre:87}
G.~Greiner.
\newblock \emph{Perturbing the boundary conditions of a generator}.
\newblock Houston J. Math. \textbf{13} (1987), 213--229.
\newblock \urlprefix\url{https://dx.doi.org/10.1007/s00233-011-9361-3}.

\bibitem[GTK19]{GTK:19}
P.~Gwi\.{z}d\.{z} and M.~Tyran-Kami\'{n}ska.
\newblock \emph{Positive semigroups and perturbations of boundary conditions}.
\newblock Positivity \textbf{23} (2019), 921--939.
\newblock \urlprefix\url{https://doi.org/10.1007/s11117-019-00644-w}.

\bibitem[Haa07]{Haa:07}
M.~Haase.
\newblock \emph{Convexity inequalities for positive operators}.
\newblock Positivity \textbf{11} (2007), 57--68.
\newblock \urlprefix\url{https://doi.org/10.1007/s11117-006-1975-4}.

\bibitem[LM72]{LM:72}
J.-L. Lions and E.~Magenes.
\newblock \emph{Non-Homogeneous Boundary Value Problems and Applications.
  {V}ol. {I}}.
\newblock Springer-Verlag, New York-Heidelberg (1972).
\newblock \urlprefix\url{https://dx.doi.org/10.1007/978-3-642-65161-8}.
\newblock Translated from the French by P. Kenneth, Die Grundlehren der
  mathematischen Wissenschaften, Band 181.

\bibitem[Nag86]{Nag(e):86}
R.~Nagel, ed.
\newblock \emph{One-parameter {S}emigroups of {P}ositive {O}perators},
  \emph{Lect. Notes in Math.}, vol. 1184.
\newblock Springer-Verlag (1986).
\newblock \urlprefix\url{https://doi.org/10.1007/BFb0074922}.

\bibitem[Sch80]{Sch:80}
H.~H. Schaefer.
\newblock \emph{Topological {V}ector {S}paces}, \emph{Graduate Texts in Math.},
  vol.~3.
\newblock Springer-Verlag (1980).
\newblock \urlprefix\url{https://doi.org/10.1007/978-1-4684-9928-5}.

\bibitem[Sta05]{Sta:05}
O.~Staffans.
\newblock \emph{Well-Posed Linear Systems}, \emph{Encyclopedia of Mathematics
  and its Applications}, vol. 103.
\newblock Cambridge University Press, Cambridge (2005).
\newblock \urlprefix\url{https://dx.doi.org/10.1017/CBO9780511543197}.

\bibitem[Tay96]{Tay:96}
M.~Taylor.
\newblock \emph{{Partial} {Differential} {Equations} {I}. {Basic} {Theory}},
  \emph{Appl. Math. Sci.}, vol. 115.
\newblock Springer-Verlag (1996).
\newblock \urlprefix\url{https://doi.org/10.1007/978-3-031-33859-5}.

\bibitem[Voi89]{Voi:89}
J.~Voigt.
\newblock \emph{On resolvent positive operators and positive strongly
  continuous semigroups on {AL}-spaces}.
\newblock Semigroup Forum \textbf{38} (1989), 263--266.
\newblock \urlprefix\url{https://doi.org/10.1007/BF02573236}.

\bibitem[Wei94]{Wei:94a}
G.~Weiss.
\newblock \emph{Regular linear systems with feedback}.
\newblock Math. Control Signals Systems \textbf{7} (1994), 23--57.
\newblock \urlprefix\url{https://dx.doi.org/10.1007/BF01211484}.

\end{thebibliography}
\end{document}